\documentclass[12pt]{article}
\usepackage{ct}

\usepackage{tikz}\usetikzlibrary{cd,calc}

\specs{}{}{}{}

\dateline{}{}{}



\title{Polynomial and combinatorial analogues\\of Gauss congruence}

\author[1]{Fern Gossow}

\affil[1]{%
Department of Mathematics and Statistics, University of Sydney, NSW, Australia

\email{F.Gossow@maths.usyd.edu.au}%
}

\definecolor{blu}{HTML}{5BCEFA}
\definecolor{pnk}{HTML}{F5A9B8}
\definecolor{ppl}{HTML}{9C59D1}
\definecolor{ylw}{HTML}{FCF434}

\newcommand\Z{\mathbb Z}
\newcommand\maj{\mathrm{maj}}

\newcommand\rk{\mathsf{rk}}
\newcommand\eps{\varepsilon}
\newcommand\Cyc{\mathsf{C}}
\newcommand\tr{\mathrm{tr}}
\newcommand\x{{\mathbf x}}
\newcommand{\qbinom}[2]{\genfrac{[}{]}{0pt}{}{#1}{#2}_q}

\newcommand{\tube}[6]{($(#5:#3-#4)+(#1,#2)$) arc (#5:#6:#3-#4) -- 
($(#6:#3-#4)+(#1,#2)$) arc (#6+180:#6:#4) -- 
($(#6:#3+#4)+(#1,#2)$) arc (#6:#5:#3+#4) -- 
($(#5:#3+#4)+(#1,#2)$) arc (#5:#5-180:#4) --
cycle}
\newcommand{\ortharc}[5]{($(#1,#2)+(#4:#3)$) arc (#4-90:#5-270:{tan((#5-#4)/2)})}

\begin{document}

\maketitle


\begin{abstract}
The cyclic sieving phenomenon provides a link between a polynomial analogue of Gauss congruence known as $q$-Gauss congruence, and a combinatorial analogue of Gauss congruence based on sequences of cyclic group actions. We strengthen this link in two major ways: by characterising $q$-Gauss congruence via explicit formulae, and by developing a universal model for the combinatorics based on necklaces which allow beads to vary in both colour and length. This gives many novel examples of cyclic sieving involving necklaces, path walks, tubings and more. We extend the definition of Gauss congruence to sequences indexed by an arbitrary ranked semigroup, and synthesise known results into this theory.
\end{abstract}



\section{Introduction}

An integer sequence $(a_n)_{n\geq 1}$ satisfies \emph{Gauss congruence} if
\[\sum_{d\mid n}\mu(n/d)a_{d}\equiv 0\pmod{n}\]
for every $n\geq 1$, where $\mu$ is the number-theoretic M\"obius function
\[\mu(n)=\begin{cases}
(-1)^r&\text{if $n$ is a product of $r$ distinct primes,}\\
0 & \text{otherwise}.
\end{cases}\]
These sequences are called Dold sequences, Fermat sequences, Newton sequences and other names. They appear in the study of topological dynamics and algebraic number theory (see \cite{dold-survey} and \cite{zarelua} for relevant surveys). They include $(\lambda^n)$ for $\lambda\in\Z$, the divisor-sum function $(\sigma(n))$ and the Lucas numbers $(L_n)$.

We obtain by M\"obius inversion and \cite[Theorems 5,6]{gauss-newton2} two useful parametrisations of Gauss congruences by arbitrary integer sequences $(b_n)$ and $(c_n)$. This association is used implicitly throughout this paper.
\begin{theorem}\label{thm:abc-intro}Let $(a_n)$ be a sequence of integers. The following are equivalent:
\begin{itemize}
\item The sequence $(a_n)$ satisfies Gauss congruence.
\item There exists a (unique) integer sequence $(b_n)$ such that for every $n\geq 1$:
\[a_n=\sum_{d\mid n}b_dd\]
\item There exists a (unique) integer sequence $(c_n)$ such that for every $n\geq 1$:
\[a_n=\sum_{n_1+n_2+\cdots=n}n_1c_{n_1}c_{n_2}\cdots.\]
\end{itemize}
\end{theorem}

Gorodetsky in \cite{q-congruences} defines a polynomial analogue for Gauss congruence. A sequence of integer polynomials $(f_n(q))_{n\geq 1}$ satisfies $q$-\emph{Gauss congruence} if
\[\sum_{d\mid n}\mu(n/d)f_d(q^{n/d})\equiv 0\pmod{[n]_q}\]
for every $n\geq 1$, where $[n]_q:=1+q+\cdots+q^{n-1}$. If $(f_n(q))$ satisfies $q$-Gauss congruence, then $(f_n(1))$ satisfies Gauss congruence. Conversely, an integer polynomial sequence $(g_n(q))$ with $f_n(1)=g_n(1)$ satisfies $q$-Gauss congruence if and only if $f_n(q)\equiv g_n(q)\pmod{q^n-1}$ for every $n$. Hence, to determine all $q$-Gauss congruences it is enough to find an appropriate $q$-analogue for every Gauss congruence.

We achieve this in three ways, mirroring Theorem \ref{thm:abc-intro}. Let $\mu_j$ denote the $j^\text{th}$ Ramanujan sum (see Definition \ref{def:ramanujan}) and let $[\exp_{\lambda}(n)]_q$ be an appropriate $q$-analogue of the Gauss congurence $(\lambda^n)_{n\geq 1}$ for $\lambda\in\Z$ (see Definition \ref{def:exp}).
\begin{theorem}\label{thm:abc-q-intro}
Fix a sequence $(a_n)$ satisfying Gauss congruence. The following sequences $(f_n(q))$ satisfy $q$-Gauss congruence and $f_n(1)=a_n$ for every $n$:
\begin{align*}f_n(q)&=\frac{1}{n}\sum_{j=0}^{n-1}\left(\sum_{d\mid n}\mu_j(n/d)a_d\right)q^j,\\
f_n(q)&=\sum_{d\mid n}[d]_{q^{n/d}}b_d,\text{ and}\\
f_n(q)&=\sum\frac{[n]_q}{[n_1+n_2+\cdots]_q}\qbinom{n_1+n_2+\cdots}{n_1,n_2,\dots}\prod_{m}[\exp_{c_m}(n_m)]_q\end{align*}
where the final sum is over all $n_1,n_2,\dots\geq 0$ with $n_1+2n_2+3n_3+\cdots=n$.
\end{theorem}

We define a combinatorial analogue of Gauss congruence in Section \ref{sect:lyndon} by combining the work of \cite{lyndon} and \cite{q-congruences}. Suppose $(X_n)_{n\geq 1}$ is a sequence of finite sets, where each $X_n$ is equipped with an action of the cyclic group $\Cyc_n$. We call this sequence a \emph{Lyndon structure} if
\[\#X_n^{\Cyc_d}=\#X_{n/d}\]
for every $n\geq 1$ and $d\mid n$, where $X_n^{\Cyc_d}$ denotes the subset of $X_n$ fixed by every element of the subgroup $\Cyc_d\leq\Cyc_n$. 

Recall from \cite{csp} the notion of a cyclic sieving triple: given a finite set $X$ with an action of the cyclic group $\Cyc_n$ and an integer polynomial $f(q)$, the triple $(X,\Cyc_n,f(q))$ satisfies the cyclic sieving phenomenon (CSP) if
\[f(\omega_d)=\#X^{\Cyc_d}\]
for every $d\mid n$, where $\omega_d$ is a primitive $d^\text{th}$ root of unity.

\begin{theorem}
Fix a Lyndon structure $(X_n)$, and set $a_n:=\#X_n$. The sequence $(a_n)$ satisfies Gauss congruence, and $(b_n)$ counts orbits of maximal order in $X_n$. If $(f_n(q))$ is a sequence of integer polynomials with $f_n(1)=a_n$, the triples $(X_n,\Cyc_n,f_n(q))$ satisfy the cyclic sieving phenomenon for every $n\geq 1$ if and only if $(f_n(q))$ satisfies $q$-Gauss congruence.
\end{theorem}

Hence, for every Lyndon structure $(X_n)$ we can recover a family of cyclic sieving triples by determining a $q$-Gauss congruence which enumerates $(X_n)$ at $q=1$. For instance, in Section \ref{sect:tubings} we consider tubings on the cycle graph $\Gamma_n$, which give a combinatorial analogue for the central Delannoy numbers.
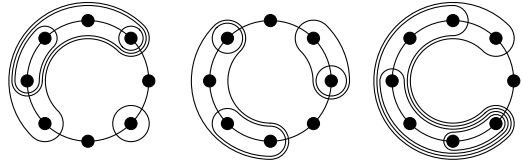
\[\begin{tikzpicture}[scale = 0.8]
\draw \tube{0}{0}{1}{0.3}{1*45}{5*45};
\draw \tube{0}{0}{1}{0.25}{1*45}{4*45};
\draw \tube{0}{0}{1}{0.2}{1*45}{1*45};
\draw \tube{0}{0}{1}{0.2}{3*45}{4*45};
\draw \tube{0}{0}{1}{0.3}{7*45}{7*45};

\draw \tube{3}{0}{1}{0.3}{3*45}{6*45};
\draw \tube{3}{0}{1}{0.25}{5*45}{6*45};
\draw \tube{3}{0}{1}{0.25}{3*45}{3*45};
\draw \tube{3}{0}{1}{0.3}{0*45}{1*45};
\draw \tube{3}{0}{1}{0.25}{0*45}{0*45};

\draw \tube{6}{0}{1}{0.3}{1*45}{7*45};
\draw \tube{6}{0}{1}{0.25}{2*45}{7*45};
\draw \tube{6}{0}{1}{0.2}{4*45}{7*45};
\draw \tube{6}{0}{1}{0.15}{6*45}{7*45};

\foreach \x in {0,1,2} {
	\draw (3*\x,0) circle (1);
}
\foreach \x in {0,1,2} {
	\foreach \ang in {0,1,2,3,4,5,6,7} {
		\draw [fill = black] ($(3*\x,0)+(45*\ang:1)$) circle (0.1);
	}
}
\end{tikzpicture}\]
\begin{theorem}\label{thm:tubings-intro}
Let $X_n$ be the set of (improper) tubings on $\Gamma_n$. Define the polynomials
\[f_n(q)=\sum_{k=0}^{n-1}\qbinom{n+k-1}{k}\qbinom{n-1}{k}.\]
Then $(X_n)$ is a Lyndon structure and $(f_n(q))$ satisfies $q$-Gauss congruence and $f_n(1)=\#X_n$. Hence, $(X_n,\Cyc_n,f_n(q))$ satisfies the cyclic sieving phenomenon for every $n\geq 1$.
\end{theorem}

Lyndon structures are common in the study of necklace-type objects. In Section \ref{sect:festoons} we define \emph{festoons}: necklaces in which beads vary in both colour and length. If $X_n$ is the set of festoons of total length $n$ where beads of length $i$ can be coloured in $c_i$ ways, then $(X_n)$ is a Lyndon structure whose enumeration sequence $(a_n)$ is associated to $(c_n)$ as in Theorem \ref{thm:abc-intro}. Hence, any Gauss congruence is enumerated by festoons on a given set of beads, although we require a signed count when $c_n$ takes negative values (see Section \ref{sect:signed}).

For example, when $c_1=2$, $c_3=1$ and $c_n=0$ otherwise, possible elements of $X_9$ include:
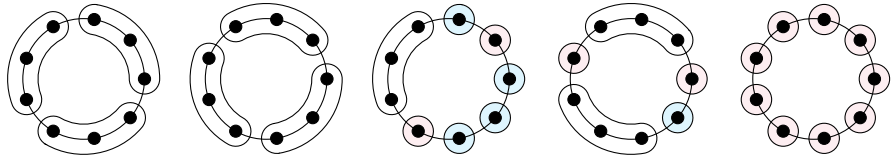
\[\begin{tikzpicture}[scale = 0.8]
\draw \tube{0}{0}{1}{0.25}{0*40}{2*40};
\draw \tube{0}{0}{1}{0.25}{3*40}{5*40};
\draw \tube{0}{0}{1}{0.25}{6*40}{8*40};

\draw \tube{3}{0}{1}{0.25}{1*40}{3*40};
\draw \tube{3}{0}{1}{0.25}{4*40}{6*40};
\draw \tube{3}{0}{1}{0.25}{7*40}{9*40};

\draw [fill = blu!20] \tube{6}{0}{1}{0.25}{0*40}{0*40};
\draw [fill = pnk!20] \tube{6}{0}{1}{0.25}{1*40}{1*40};
\draw [fill = blu!20] \tube{6}{0}{1}{0.25}{2*40}{2*40};
\draw \tube{6}{0}{1}{0.25}{3*40}{5*40};
\draw [fill = pnk!20] \tube{6}{0}{1}{0.25}{6*40}{6*40};
\draw [fill = blu!20] \tube{6}{0}{1}{0.25}{7*40}{7*40};
\draw [fill = blu!20] \tube{6}{0}{1}{0.25}{8*40}{8*40};

\draw [fill = pnk!20] \tube{9}{0}{1}{0.25}{0*40}{0*40};
\draw \tube{9}{0}{1}{0.25}{1*40}{3*40};
\draw [fill = pnk!20] \tube{9}{0}{1}{0.25}{4*40}{4*40};
\draw \tube{9}{0}{1}{0.25}{5*40}{7*40};
\draw [fill = blu!20] \tube{9}{0}{1}{0.25}{8*40}{8*40};

\foreach \ang in {0,1,2,3,4,5,6,7,8} {
	\draw [fill = pnk!20] ($(12,0)+(40*\ang:1)$) circle (0.25);
}

\foreach \x in {0,1,2,3,4} {
	\draw (3*\x,0) circle (1);
}
\foreach \x in {0,1,2,3,4} {
	\foreach \ang in {0,1,2,3,4,5,6,7,8} {
		\draw [fill = black] ($(3*\x,0)+(40*\ang:1)$) circle (0.1);
	}
}
\end{tikzpicture}\]
For festoons of a single colour with given content, we obtain the following result.
\begin{theorem}\label{thm:festoon-intro}
Let $X$ be the set of festoons where $n_i$ beads have length $i$ for every $i\geq 1$. Set $n:=n_1+n_2+\cdots$ and $N:=n_1+2n_2+3n_3+\cdots$. Then $\Cyc_N$ acts on $X$ by rotation, and if
\[f(q)=\frac{[N]_q}{[n]_q}\qbinom{n}{n_1,n_2,\dots}\]
then $(X,\Cyc_N,f(q))$ satisfies the cyclic sieving phenomenon.
\end{theorem}

Importantly, our entire theory is extend to the case where sequences are indexed not just by the positive integers, but by an arbitrary \emph{ranked semigroup} (see Section \ref{sect:semigroups} for a definition). This expanded scope allows us to view many well-known families of integers, such as the binomial coefficients, as Gauss congruences. Likewise, many well-known cyclic sieving results become examples of Lyndon structures. The main benefit to this enlarged definition is the use of novel morphism-theoretic tools for generating $q$-Gauss congruences, which we cover in Section \ref{sect:q-gauss}.

Technical proofs for some results are delayed until Section \ref{sect:proofs}.

\section{Semigroups}\label{sect:semigroups}

We give the necessary definitions and examples of semigroups which will be used frequently throughout the paper to index sequences.

A \emph{commutative semigroup} is a pair $(S,+)$ consisting of a set $S$ and an associative, commutative operation $+:S\times S\to S$. We usually write $S$ as shorthand for $(S,+)$. A \emph{morphism} of commutative semigroups is a map $\phi:S\to T$ such that $\phi(s_1+s_2)=\phi(s_1)+\phi(s_2)$ for every $s_1,s_2\in S$. Let $\Z_{\geq 1}$ denote the semigroup of positive integers with the usual addition.

\begin{definition}
A \emph{ranked semigroup} is a commutative semigroup $S$ with a specified morphism $\rk:S\to\Z_{\geq 1}$, called the \emph{rank} function.
\end{definition}
\begin{remark}
Such semigroups are known as metrical semigroups (in \cite{metrical}) or semigroups with length (in \cite{howie}). They have no additive identity.
\end{remark}

A semigroup $S$ has \emph{finite decomposition} if there are only finitely many ways to express each $s\in S$ in the form $s=s_1+\cdots+s_k$ for $s_i\in S$. All finitely-generated ranked semigroups have this property.

Fix a ranked semigroup $S$ and $s,t\in S$. Define $s-t:=\{u\mid u+t=s\}$. We define a partial order on $S$ by setting $t\leq s$ if $s=t$ or $s-t$ is nonempty. We call $S$ \emph{cancellative} if $|s-t|\leq 1$ for every $s,t\in S$.

For $d\geq 1$, define $d\cdot s:=s+s+\cdots+s$ with $d$ summands. If $s=d\cdot t$ for some $s,t\in S$, we write $t\mid s$ and set $s/t:=\rk(s)/\rk(t)=d$. The relation $|$ is a partial order on $S$. We also define $s/d:=\{t\mid d\cdot t=s\}$ and write $d\mid s$ if $s/d$ is nonempty. A semigroup is \emph{torsion-free} if $|s/d|\leq 1$ for every $s\in S$ and $d\geq 1$.

Our first example is the free ranked semigroup on $A$, where the rank of each element is determined by a function $\psi$ (although we allow $\psi$ to take arbitrary integer values).

\begin{example}\label{ex:dual}
Fix an arbitrary set $A$ and a function $\psi:A\to\Z$. For $\alpha:A\to\Z_{\geq 0}$, set
\[\psi(\alpha):=\sum_{a\in A}\psi(a)\alpha(a)\quad\text{and}\quad|\alpha|:=\sum_{a\in A}\alpha(a).\]
Let $A^\psi$ be the ranked semigroup of functions
\[\{\alpha:A\to \Z_{\geq 0}\mid 0<|\alpha|,\psi(\alpha)<\infty\}\]
with pointwise addition and rank function $\rk(\alpha)=\psi(\alpha)$.

As a special case, define $*:A\to\Z$ by $a\mapsto 1$ for every $a\in A$. Then $*(\alpha)=|\alpha|$, so $A^*$ is the ranked semigroup of functions with $0<|\alpha|<\infty$, and $\rk(\alpha)=|\alpha|$.
\end{example}

The ranked semigroup $A^\psi$ is cancellative, torsion-free and has finite decomposition. For $d\geq 1$ and $\alpha\in A^\psi$, we have $d\mid \alpha$ if and only if $d\mid \alpha(a)$ for every $a\in A$, in which case $(\alpha/d)(a)=\alpha(a)/d$. If $A=\{1\}$ and $\psi(1)=1$, then $A^\psi\cong\Z_{\geq 1}$.

\begin{definition}\label{def:chains}
Let $S$ be a ranked semigroup and let $T$ be a commutative semigroup. Define the ranked semigroup
\[S[T]:=\{(s,t)\mid s\in S,\,t\in T\}\]
with $(s_1,t_1)+(s_2,t_2):=(s_1+s_2,t_1+t_2)$ and $\rk(s,t):=\rk(s)$.
\end{definition}

We call this operation \emph{chaining}. If $S$ and $T$ are ranked semigroups and $U$ is a commutative semigroup, then $S[T][U]\cong S[T[U]]$ under $((s,t),u)\mapsto (s,(t,u))$, so we write this element as $(s,t,u)$. Likewise, we have $S[T][U]\cong S[U][T]$ under $(s,t,u)\mapsto (s,u,t)$.

\section{Gauss congruence}

We extend Gauss congruence to sequences $(-_s)$ indexed by a ranked semigroup $S$ and state the analogue of Theorem \ref{thm:abc-intro}. We generalise a number of known identities relating Gauss congruences to their parameter sequences.

\subsection{Definition}

\begin{definition}\label{def:gauss-cong}An integer sequence $(a_s)$ satisfies \emph{Gauss congruence} (on $S$) if
\[\sum_{t\mid s}\mu(s/t)a_t\equiv 0\pmod{\rk(s)}\]
for every $s\in S$.\end{definition}
\begin{remark}\label{rmk:gauss-poset}
Gauss congruence was defined on arbitrary posets in \cite{dold-poset}, and viewing $S$ as a poset under division recovers the above definition. For more information, see the proof of Theorem \ref{thm:abc} in Section \ref{sect:proofs}.
\end{remark}

\subsection{Characterisation}

\begin{theorem}\label{thm:abc}Let $S$ be a ranked semigroup with finite decomposition and $(a_s)$ a sequence of integers indexed by $S$. The following are equivalent:
\begin{itemize}
\item The sequence $(a_s)$ satisfies Gauss congruence.
\item There exists an integer sequence $(b_s)$ such that for every $s\in S$:
\begin{equation}\label{eq:abS}a_s=\sum_{t\mid s}b_t\rk(t)\end{equation}
\item There exists an integer sequence $(c_s)$ such that for every $s\in S$:
\begin{equation}\label{eq:ac1S}a_s=\sum_{s_1+s_2+\cdots=s}\rk(s_1)c_{s_1}c_{s_2}\cdots.\end{equation}
\end{itemize}
\end{theorem}

\begin{example}
Fix a formal power series $D(t)=d_0+d_1t+d_2t^2+\cdots$ with integer coefficients. The sequence $(a_{n,k})$ given by
\[a_{n,k}:=[t^k]D(t)^n\]
satisfies Gauss congruence on $\Z_{\geq 1}[\Z]$. The associated sequence $(c_{n,k})$ is given by $c_{1,k}:=d_k$ and $c_{n,k}:=0$ for $n>1$. When $D(t)=1+t$, $1/(1-t)$ and $t/(1-t)$ respectively we obtain the Gauss congruences
\[a_{n,k}=\binom{n}{k},\qquad a_{n,k}=\binom{n+k-1}{k},\quad\text{and}\quad a_{n,k}=\binom{k-1}{k-n}\]
where the binomial coefficient is zero outside of its usual range.
\end{example}

By \cite[Theorem 44]{dold-poset} (also see the proof of Theorem \ref{thm:abc}), the M\"obius function in the definition of Gauss congruence can be replaced with other number theoretic functions such as Euler's totient $\phi$.

\begin{proposition}\label{prop:gauss-other}
Let $\varphi:\Z_{\geq 1}\to\Z$ be a function such that $\varphi(1)=\pm 1$ and 
\[\sum_{d\mid n}\varphi(n/d)\equiv 0\pmod{n}\]
for every $n\geq 1$. Then $(a_s)$ satisfies Gauss congruence if and only if
\[\sum_{t\mid s}\varphi(s/t)a_t\equiv 0\pmod{\rk(s)}\]
for every $s\in S$.
\end{proposition}

\subsection{Identities}\label{sect:identities}

We generalise identities between the sequences $(a_n)$, $(b_n)$ and $(c_n)$ in the case $S=\Z_{\geq 1}$ appearing in \cite{gauss-newton2} and \cite{gauss-newton1}.

For each summand of (\ref{eq:ac1S}), define a function $\alpha:S\to\Z_{\geq 0}$ by setting $\alpha(t)$ as the number of occurences of $t$. Summing over all relevant such functions gives
\begin{equation}\label{eq:ac2S}a_s=\sum_{\substack{\alpha\in S^*\\\tau(\alpha)=s}}\frac{\rk(s)}{|\alpha|}\binom{|\alpha|}{\alpha}\prod_{t\in S} c_t^{\alpha(t)}\end{equation}
where $\tau:S^*\to S$ is given by $\tau(\alpha)=\sum_{s\in S}\alpha(s)s$, and
\[\binom{|\alpha|}{\alpha}:=\frac{(\sum_{a\in \alpha}\alpha(a))!}{\prod_{a\in A}\alpha(a)!}\]
for $\alpha\in A^\psi$ (or in this case, $S^*$).

Alternatively, we can rewrite (\ref{eq:ac1S}) as an almost-linear recurrence by `peeling off' the final term in each summand. Then
\begin{equation}\label{eq:ac-recS}a_s=\rk(s)c_s+\sum_{t+u=s}c_ta_u=\rk(s)c_s+\sum_{t<s}c_ta_{s-t}\end{equation}
where $a_{s-t}:=\sum_{u\in s-t}a_u$. 

For $S=\Z_{\geq 1}$, we can rewrite (\ref{eq:ac1S}) as
\begin{equation}\label{eq:M}a_n=\tr\begin{bmatrix}
c_1 & c_2 & c_3 & \cdots \\
1 & 0 & 0 \\
0 & 1 & 0 \\
\vdots & & \ddots
\end{bmatrix}^n\end{equation}
which can be used to prove that $a_n=\tr(M^n)$ satisfies Gauss congruence for any integer matrix $M$. We also have a direct relationship between $(b_n)$ and $(c_n)$ expressed as an equality of formal power series
\begin{equation}\label{eq:bc}\sum_{n\geq 1}c_nx^n=1-\prod_{n\geq 1}(1-x^n)^{b_n}.\end{equation}

Versions of these identities can be generalised by working over the ring of infinite $\Z$-linear combinations of symbols $\{x^s\mid s\in S\cup\{0\}\}$ with $x^sx^t=x^{s+t}$ for $s,t\in S$ and where $1:=x^0$ is the mutliplicative identity. This perspective in the case $S=\{1,\dots,r\}^*$ was studied in \cite{multivar}, although a different generalisation of Gauss congruence is used.

\section{q-Gauss congruence}\label{sect:q-gauss}

As with Gauss congruence, we extend the notion of $q$-Gauss congruence from \cite{q-congruences} to ranked semigroups. We build a number of morphism-theoretic tools for constructing these congruences. The main result is Theorem \ref{thm:abc-q}, which gives a number of explicit characterisations of $q$-Gauss congruences. This theory is foundational to finding examples of cyclic sieving results in Section \ref{sect:lyndon} onwards.

\subsection{Setup}

We work with sequences indexed by a fixed ranked semigroup $S$.

\begin{definition}\label{def:q-gauss}
A sequence of integer polynomials $(f_s(q))$ satisfies $q$-\emph{Gauss congruence} (on $S$) if
\[\sum_{t\mid s}\mu(s/t)f_t(q^{s/t})\equiv 0\pmod{[\rk(s)]_q}\]
for every $s\in S$.
\end{definition}

Gorodetsky in \cite{q-congruences} gave an extremely useful arithmetic characterisation of $q$-Gauss congruences, which holds in our generality. Recall that $\omega_d$ denotes a primitive $d^\text{th}$ root of unity.

\begin{proposition}\label{prop:root-sub}
A sequence of integer polynomials $(f_s(q))$ satisfies $q$-Gauss congruence if and only if
\[f_s(\omega_d)=\sum_{t\in s/d}f_t(1)\]
for every $s\in S$ and $d\mid\rk(s)$.
\end{proposition}

Before giving examples of $q$-Gauss congruences, we provide a number of tools for translating congruences between ranked semigroups. Proofs are given in Section \ref{sect:proofs}, and make heavy use of Proposition \ref{prop:root-sub}. We refer to these operations by their names throughout the paper.

\begin{definition}
A morphism of ranked semigroups $\phi:S\to S'$ is;
\begin{itemize}
\item \emph{rank-multiplying} if $\rk(s)\mid \rk(\phi(s))$ and
\item \emph{rank-dividing} if $\rk(\phi(s))\mid\rk(s)$
\end{itemize}
for every $s\in S$.
\end{definition}

\begin{lemma}[Pushforward]\label{lem:pushforward}
Suppose $\phi:S\to T$ is a rank-dividing morphism such that $\phi^{-1}(t)$ is finite for every $t\in T$. If $(f_s(q))$ satisfies  $q$-Gauss congruence on $S$, then
\[g_t(q):=\sum_{s\in\phi^{-1}(t)}f_s(q)\]
satisfies $q$-Gauss congruence on $T$.
\end{lemma}

Taking $\phi$ to be the rank function $\rk:S\to\Z_{\geq 1}$, we obtain that
\[g_n(q)=\sum_{\substack{s\in S\\\rk(s)=n}}f_s(q)\]
gives a $q$-Gauss congruence on $\Z_{\geq 1}$ whenever $(f_s(q))$ is a $q$-Gauss congruence on $S$.

\begin{lemma}[Pullback]\label{lem:pullback}
Suppose $\phi:S\to T$ is a rank-multiplying morphism such that $\phi$ restricts to a bijection $\phi:s/d\to\phi(s)/d$ for every $s\in S$ and $d\mid\rk(s)$. If $(g_t(q))$ satisfies $q$-Gauss congruence on $T$, then
\[f_s(q):=g_{\phi(s)}(q)\]
satisfies $q$-Gauss congruence on $S$.
\end{lemma}
\begin{remark}
When $S$ is torsion-free, the bijection condition on $\phi$ is equivalent to the following: if $d\mid\rk(s)$ and $d\mid\phi(s)$, then $d\mid s$.
\end{remark}

Taking $S'$ to be a subset of $S$ closed under addition and division, using the inclusion morphism $\iota:S'\to S$ we obtain that $(f_s(q))$ satisfies $q$-Gauss congruence on $S'$ whenever it does on $S$.

Finally, we consider multiplying $q$-Gauss congruences on torsion-free semigroups. We extend this to chains of semigroups from Definition \ref{def:chains}.

\begin{lemma}\label{lem:multiply}
If $S$ is torsion-free and $(f_s(q))$ and $(g_s(q))$ are $q$-Gauss congruences on $S$, then so is $(f_s(q)g_s(q))$.
\end{lemma}
\begin{remark}
The proof allows a slight relaxation of the condition on $(g_s(q))$: we only need that $g_s(\omega_d)=g_{s/d}(1)$ for every $d\mid s$, rather than for every $d\mid \rk(s)$.
\end{remark}

\begin{lemma}[Prefix Chaining]\label{lem:chaining-pre}
Suppose $S$, $T$ and $U$ are torsion-free semigroups, and $S$ is ranked. If $(f_{s,t}(q))$ and $(g_{s,u}(q))$ satisfy $q$-Gauss congruence on $S[T]$ and $S[U]$ respectively, then
\[h_{s,t,u}(q):=f_{s,t}(q)g_{s,u}(q)\]
satisfies $q$-Gauss congruence on $S[T][U]$.
\end{lemma}

\begin{lemma}[Suffix Chaining]\label{lem:chaining-suf}
Suppose $S$, $T$ and $U$ are torsion-free semigroups, and $S$ and $T$ are ranked. If $(f_{s,t}(q))$ and $(g_{t,u}(q))$ satisfy $q$-Gauss congruences on $S[T]$ and $T[U]$ respectively, then
\[h_{s,t,u}(q):=f_{s,t}(q)g_{t,u}(q)\]
satisfies $q$-Gauss congruence on $S[T][U]$.
\end{lemma}

\subsection{Examples}

Recall the free ranked semigroup $A^\psi$ from Example \ref{ex:dual}. For $\alpha\in A^\psi$, define the $q$-multinomial coefficient
\[\qbinom{|\alpha|}{\alpha}:=\frac{[|\alpha|]_q\cdots [2]_q[1]_q}{\prod_{a\in A}[\alpha(a)]_q\cdots [2]_q[1]_q}.\]
We extend this definition to all functions $\alpha:A\to\Z$ with finite support by setting the coefficient to be zero when any $\alpha(s)$ is negative. A special case is the $q$-binomial coefficient
\[\qbinom{n}{k}:=\frac{[n]_q\cdots [2]_q[1]_q}{[k]_q\cdots [1]_q[n-k]_q\cdots[1]_q}.\]
\begin{theorem}\label{thm:fund}
The sequence $(f_\alpha(q))$ on $A^\psi$ given by
\[f_\alpha(q)=\frac{[\psi(\alpha)]_q}{[|\alpha|]_q}\qbinom{|\alpha|}{\alpha}\]
satisfies $q$-Gauss congruence.
\end{theorem}
\begin{remark}When $\psi:A\to\Z_{\geq 1}$, the Gauss conguence $(f_\alpha(1))$ on $A^\psi$ satisfies $c_\alpha=1$ when $|\alpha|=1$ and $c_\alpha=0$ otherwise.\end{remark}

For $A=\{0,1\}$, define $\phi:A^*\hookrightarrow\Z_{\geq 1}[\Z]$ by $\alpha\mapsto(\alpha(0)+\alpha(1),\alpha(0))$. Then the preimage of $(n,k)$ under $\phi$ contains the single element with $\alpha(0)=k$ and $\alpha(1)=n-k$ when $n\geq 1$ and $0\leq k\leq n$, and is empty otherwise. Pushing forward along $\phi$, we obtain that
\[f_{n,k}(q)=\qbinom{n}{k}\]
gives a $q$-Gauss congruence $(f_{n,k}(q))$ on $\Z_{\geq1}[\Z]$. Similarly, setting $\psi(0)=0$, $\psi(1)=1$ and pushing foward along $\phi:A^\psi\hookrightarrow\Z_{\geq 1}[\Z]$ with $\alpha\mapsto (\alpha(1),\alpha(0))$, we obtain
\[f_{n,k}(q)=\qbinom{n+k-1}{k}.\]
We now give a $q$-Gauss congruence associated to $(\lambda^n)$ for any $\lambda\in\Z$, which satisfies Gauss congruence with $c_1=\lambda$ and $c_n=0$ otherwise. One can check by computation that
\[\eps_n(q):=\begin{cases}
-1 & \text{if $n$ odd},\\
q^{n/2} & \text{if $n$ even}\end{cases}\]
gives a $q$-Gauss congruence $(\eps_n(q))$ on $\Z_{\geq 1}$ and that $\eps_n(1)=(-1)^n$.

\begin{definition}\label{def:exp}
Fix $\lambda\geq 1$. For $n\geq 1$, define
\[[\exp_\lambda(n)]_q:=\sum_{n_1+\cdots+n_\lambda=n}\qbinom{n}{n_1,\dots,n_\lambda}\]
and $[\exp_{-\lambda}(n)]_q:=\eps_n(q)[\exp_\lambda(n)]_q$. We also set $[\exp_0(n)]_q=0$.
\end{definition}
\begin{proposition}\label{prop:exp}
Fix $\lambda\in\Z$ and let $f_n(q):=[\exp_\lambda(n)]_q$. Then $f_n(1)=\lambda^n$ and $(f_n(q))$ satisfies $q$-Gauss congruence on $\Z_{\geq 1}$.
\end{proposition}

\subsection{Characterisation}

Before stating the main theorem, we define a number-theoretic function known as \emph{Ramanujan's sum} (often denoted $c_d(j)$). Our use of these sums is inspired by \cite{csp}.

\begin{definition}\label{def:ramanujan}
For $d\geq 1$ and $j\in\Z$, let $\mu_j(d):=\sum_\omega \omega^j$ where $\omega$ ranges over all primitive $d^\text{th}$ roots of unity. Note that $\mu_1(d)=\mu(d)$ for every $d\geq 1$.
\end{definition}

\begin{theorem}\label{thm:abc-q}
Fix a ranked semigroup $S$ with finite decomposition and a Gauss congruence $(a_s)$. The following give $q$-Gauss congruences $(g_s(q))$ on $S$ and satisfy $g_s(1)=a_s$:
\begin{align*}
g_s(q)&=\frac{1}{\rk(s)}\sum_{j=0}^{\rk(s)-1}\left(\sum_{t\mid s}\mu_j(s/t)a_t\right)q^j,\\
g_s(q)&=\sum_{t\mid s}[\rk(t)]_{q^{s/t}}b_t,\text{ and}\\
g_s(q)&=\sum_{\substack{\alpha\in S^*\\\tau(\alpha)=s}}\frac{[\rk(s)]_q}{[|\alpha|]_q}\qbinom{|\alpha|}{\alpha}\prod_{t}[\exp_{c_t}(\alpha(t))]_q.\end{align*}
Moreover, a sequence of integer polynomials $(f_s(q))$ on $S$ satisfies $q$-Gauss congruence and $f_s(1)=a_s$ if and only if $f_s(q)\equiv g_s(q)\pmod{q^{\rk(s)}-1}$ for every $s\in S$, where $(g_s(q))$ is any of the above.
\end{theorem}

\begin{remark}
The second formula in this Theorem resolves \cite[Problem 40]{lyndon} by connecting $(f_s(q))$ to $(b_s)$.
\end{remark}

\section{Combinatorial analogue for Gauss congruence}\label{sect:lyndon}

Fix a ranked semigroup $S$ and an integer sequence $(a_s)$ satisfying Gauss congruence. This congruence can be described by the parameter sequences $(b_s)$ or $(c_s)$ from Theorem \ref{thm:abc}, and up to a modular relation we can determine by Theorem \ref{thm:abc-q} a unique $q$-Gauss congruence $(f_s(q))$ which evaluates to $(a_s)$ at $q=1$.

A combinatorial analogue of this congruence consists of a sequence of sets $(X_s)$ enumerated by $(a_s)$, with additional combinatorial structure which gives meaning to the associated sequences $(b_s)$ and $(f_s(q))$. We call these \emph{Lyndon structures} as they exhibit `Lyndon-like cyclic sieving' as defined in \cite{lyndon} when $S=\Z_{\geq 1}$.

\begin{definition}\label{def:lyndon}
A \emph{Lyndon structure} $(X_s)$ is a sequence of finite sets where each $X_s$ is equipped with an action of the cyclic group $\Cyc_{\rk(s)}$ satisfying the following: for $d\mid\rk(s)$, the number of elements in $X_s$ fixed by $\Cyc_d\leq \Cyc_{\rk(s)}$ is
\[\#X_s^{\Cyc_d}=\sum_{t\in s/d}\#X_t.\]
\end{definition}

We call two Lyndon structures $(X_s)$ and $(Y_s)$ \emph{isomorphic} if there are bijections $X_s\xrightarrow{\sim}Y_s$ which commute with the actions of the cyclic groups.

\begin{theorem}\label{thm:lyndon}
For a Lyndon structure $(X_s)$, the sequence $(a_s)$ defined by $a_s:=\#X_s$ satisfies Gauss congruence. Moreover, $(b_s)$ counts orbits of order $\rk(s)$ in $X_s$, and $(X_s,\Cyc_{\rk(s)},f_s(q))$ satisfies the cyclic sieving phenomenon for every $s\in S$.
\end{theorem}
\begin{proof}
If $b_s$ counts orbits of order $\rk(s)$ in each $X_s$, then
\[a_s=\#X_s=\sum_{t\mid s}\rk(t)b_t\]
which proves that $(a_s)$ is the Gauss congruence associated to $(b_s)$. For $f_s(q)$ to satisfy cyclic sieving with $X_s$, we require that $f_s(\omega_d)=\sum_{t\in s/d}\#X_t=\sum_{t\in s/d}f_t(1)$ for every $d\mid\rk(s)$, which is equivalent to $q$-Gauss congruence by Lemma \ref{prop:root-sub}.
\end{proof}

\begin{remark}For a Gauss congruence $(a_s)$ to have a Lyndon structure $(X_s)$ which it enumerates, we require $b_s\geq 0$ for every $s\in S$. Theorem \ref{thm:festoon-Y} shows this is sufficient.\end{remark}

\begin{example}
Fix an ordered set $A$. For $\alpha\in A^*$, let $X_\alpha$ be the set of words with \emph{content} $\alpha$ (that is, $\alpha(a)$ occurrences of each $a\in A$) with $\Cyc_{|\alpha|}$ acting by rotation of letters. Then $(X_\alpha)$ is a Lyndon structure on $A^*$. The multinomial coefficient enumerates $X_\alpha$, and by Theorem \ref{thm:fund} its $q$-analogue satisfies $q$-Gauss congruence. Hence,
\[\left(X_\alpha,\Cyc_{|\alpha|},\qbinom{|\alpha|}{\alpha}\right)\]
satisfies the CSP for every $\alpha$. Orbits in $X_\alpha$ of maximal rank are in bijection with \emph{Lyndon words}: those which are uniquely lexicographically minimal in their multiset of rotations, and hence $b_\alpha$ counts Lyndon words of content $\alpha$.
\end{example}
\begin{remark}
The $q$-multinomial coefficient has a beautiful product-sum identity due to MacMachon (see \cite[Theorem 3.7]{andrews}). For a word $w=w_1\cdots w_{|\alpha|}\in X_\alpha$, define the \emph{major index} $\maj(w):=\sum\{i\mid w_i>w_{i+1}\}$ using the ordering on $A$. Then
\[\qbinom{|\alpha|}{\alpha}=X_\alpha^\maj(q):=\sum_{w\in X_\alpha}q^{\maj(w)}.\]
\end{remark}

A \emph{composition} is a word $w=w_1\cdots w_n$ in the ordered alphabet $\Z$ which has a \emph{length} $n$ and a \emph{sum} $k:=w_1+\cdots+w_n$. To obtain a $q$-Gauss congruence for compositions, we pushfoward $(X_\alpha^\maj(q))$ from above along the morphism 
\[\phi:\Z^*\to\Z_{\geq 1}[\Z];\qquad\alpha\mapsto \left(\sum_{n\in\Z}\alpha(n),\sum_{n\in\Z}n\alpha(n)\right).\]
The set $\bigcup_{\alpha\in\phi^{-1}(n,k)}X_\alpha$ consists of all compositions with length $n$ and sum $k$. To ensure finite preimages as required by Lemma \ref{lem:pushforward}, we restrict to subsets of $\Z$ bounded on one side.

\begin{example}
Fix a subset $A\subset\Z$ bounded from below. Let $X_{n,k}$ denote the set of compositions of length $n$ and sum $k$ from the alphabet $A$, with $\Cyc_n$ acting by rotation. Then $(X_{n,k})$ is a Lyndon structure on $\Z_{\geq 1}[\Z]$, and $(X_{n,k},\Cyc_n,X_{n,k}^\maj(q))$ satisfies the CSP for every $n\geq 1$ and $k\in\Z$.
\end{example}

We provide a number of concrete examples. In each case, we obtain an alternate $q$-Gauss congruence by finding a natural $q$-analogue for the enumeration of the compositions. We write down the implied modular equivalence with the major index polynomial.

When $A=\{0,1,2,\dots\}$, we have
\[X_{n,k}^\maj(q)\equiv\qbinom{n+k-1}{k}\pmod{q^n-1}.\]

When $A=\{1,2,3\dots\}$, we have
\[X_{n,k}^\maj(q)\equiv\qbinom{k-1}{k-n}\pmod{q^n-1}\]
where the right-hand side is obtained by pulling back the $q$-Gauss congruence from the previous example along $(n,k)\mapsto (n,k-n)$.

When $A=\{-1,0,1\}$, $(X_{n,k})$ is enumerated by the \emph{trinomial coefficients}. Letting $i$ be the number of occurences of $-1$ in the composition, we obtain
\[X_{n,k}^\maj(q)\equiv\sum_{i=0}^{\lfloor (n-k)/2\rfloor}\qbinom{n}{i}\qbinom{n-i}{k+i}\pmod{q^n-1}.\]
To see this is a $q$-Gauss congruence, we take the $q$-Gauss congruence
\[f_{n,k,i}=\qbinom{n}{k}\qbinom{k}{i}\]
on $\Z_{\geq 1}[\Z][\Z]$ obtained by chaining, and pullback along $(n,k,i)\mapsto(n,n-i,k+i)$ to obtain the $q$-Gauss congruence
\[f_{n,k,i}=\qbinom{n}{i}\qbinom{n-i}{k+i}\]
on $\Z_{\geq 1}[\Z][\Z]$. Then pushfoward along $(n,k)\mapsto (n,k,i)$ to sum over $i$. 
\begin{remark}
A third possible $q$-Gauss congruence was given in \cite[Theorem 1.2]{q-congruences}:
\[f_{n,k}(q)=\sum_{i=0}^{\lfloor (n-k)/2\rfloor}q^{i(i+b)}\qbinom{n}{i}\qbinom{n-i}{k+i}\]
for any $b\geq -1$.
\end{remark}

\section{Festoons}\label{sect:festoons}

We give an explicit construction of a Lyndon structure which enumerates a given Gauss congruence. These are based on a combinatorial object we call a \emph{festoon}, named for a type of decorative necklace. To create a festoon, decompose the cycle graph $\Gamma_n$ into intervals, and label each interval by a \emph{bead} of the corresponding length. The cyclic group $\Cyc_n$ acts on these festoons by clockwise rotation.

\begin{example}
Suppose we have the following beads: two colours of length $1$, one colour of length $2$ and one colour of length $5$. Below are some examples of festoons with length $5$.
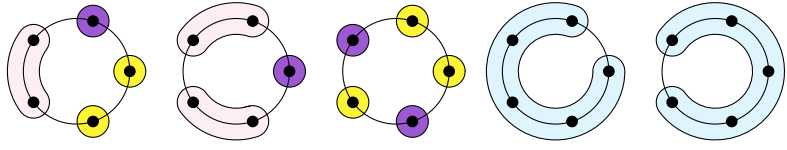
\[\begin{tikzpicture}[scale = 0.7]
\draw [fill = ylw] \tube{-3}{0}{1}{0.3}{0}{0};
\draw [fill = ppl] \tube{-3}{0}{1}{0.3}{72}{72};
\draw [fill = pnk!20] \tube{-3}{0}{1}{0.3}{2*72}{3*72};
\draw [fill = ylw] \tube{-3}{0}{1}{0.3}{4*72}{4*72};

\draw [fill = ppl] \tube{0}{0}{1}{0.3}{0}{0};
\draw [fill = pnk!20] \tube{0}{0}{1}{0.3}{72}{2*72};
\draw [fill = pnk!20] \tube{0}{0}{1}{0.3}{3*72}{4*72};

\draw [fill = ylw] \tube{3}{0}{1}{0.3}{0}{0};
\draw [fill = ylw] \tube{3}{0}{1}{0.3}{72}{72};
\draw [fill = ppl] \tube{3}{0}{1}{0.3}{2*72}{2*72};
\draw [fill = ylw] \tube{3}{0}{1}{0.3}{3*72}{3*72};
\draw [fill = ppl] \tube{3}{0}{1}{0.3}{4*72}{4*72};

\draw [fill = blu!20] \tube{6}{0}{1}{0.3}{1*72}{5*72};

\draw [fill = blu!20] \tube{9}{0}{1}{0.3}{3*72}{7*72};

\foreach \x in {-1,0,1,2,3} {
	\draw (3*\x,0) circle (1);
}
\foreach \x in {-1,0,1,2,3} {
	\foreach \ang in {0,1,2,3,4} {
		\draw [fill = black] ($(3*\x,0)+(72*\ang:1)$) circle (0.1);
	}
}
\end{tikzpicture}\]
\end{example}

\begin{remark}
A model similar to festoons has appeared in the biophysical chemistry literature under the name `multivalent bindings on the ring lattice' (for example, see \cite{binding}).
\end{remark}

\subsection{Main results}

First we consider festoons of a given content, which provides a combinatorial analogue for the $q$-Gauss congruence from Theorem \ref{thm:fund} at $q=1$, and recovers Theorem \ref{thm:festoon-intro}.

\begin{theorem}\label{thm:festoon-A}
Fix a set of beads $A$ and a length function $\psi:A\to\Z_{\geq 1}$. For $\alpha\in A^\psi$, let $X_\alpha$ be the set of festoons with content $\alpha$ (that is, with $\alpha(a)$ copies of each bead $a$). Then $(X_\alpha)$ is a Lyndon structure on $A^\psi$, and
\[\left(X_\alpha,\Cyc_{\psi(\alpha)},\frac{[\psi(\alpha)]_q}{[|\alpha|]_q}\qbinom{|\alpha|}{\alpha}\right)\]
satisfies the CSP for every $\alpha\in A^\psi$.
\end{theorem}
\begin{proof}
A festoon in $X_\alpha$ is fixed by $C_d$ if and only if it is the $d$-fold repetition of some festoon in $X_{\alpha/d}$, so $(X_\alpha)$ is a Lyndon structure. Let $f_\alpha(q)$ be the polynomials in the sieving triples, which satisfy $q$-Gauss congruence by Theorem \ref{thm:fund}. We prove $\#X_\alpha=f_\alpha(1)$. The multinomial coefficient counts possible linear arrangements of the beads with content $\alpha$. There are $\psi(\alpha)$ rotations for this arrangement around $\Gamma_{\psi(\alpha)}$, but each is counted $|\alpha|$ times, one for each cyclic rotation of the linear arrangement of these beads.
\end{proof}

We connect festoons to ranked semigroups by assigning a \emph{type} $s\in S$ to each bead. A bead with type $s$ will always have length $\rk(s)$. Given a nonnegative integer sequence $(c_s)$, a \emph{festoon coloured by} $(c_s)$ is a festoon where each bead of type $s$ is coloured in one of $c_s$ ways. If this festoon has beads of types $s_1,s_2,\dots$, we say that the type of the festoon is $s=s_1+s_2+\cdots$.

\begin{theorem}\label{thm:festoon-X}
Let $S$ be a ranked semigroup with finite decomposition, and $(c_s)$ a nonnegative integer sequence with corresponding Gauss congruence $(a_s)$. Let $X_s$ be the set of festoons of type $s$ coloured by $(c_s)$. Then $(X_s)$ is a Lyndon structure on $S$ enumerated by $(a_s)$.
\end{theorem}
\begin{proof}
A festoon of type $s$ is fixed by $C_d$ if and only if it is the $d$-fold repetition of a festoon in $X_t$ with $s=d\cdot t$, so $(X_s)$ is a Lyndon structure. They satisfy the appropriate count by (\ref{eq:ac1S}): choose parts $s_1,s_2,\dots$ with $s_1+s_2+\cdots=s$, then colour part $i$ in one of $c_{s_i}$ ways. This linear arrangement has $\rk(s_1)$ rotations where the first bead covers a distinguished vertex of $\Gamma_{\rk(s)}$.
\end{proof}

\begin{theorem}\label{thm:festoon-Y}
Let $S$ be a ranked semigroup, and $(b_s)$ a nonnegative integer sequence with corresponding Gauss congruence $(a_s)$. Let $Y_s$ be the set of festoons of type $s$ coloured by $(b_s)$, where all beads in the festoon have the same type and colour. Then $(X_s)$ is a Lyndon structure on $S$ enumerated by $(a_s)$.
\end{theorem}
\begin{proof}
By an analogous argument as above $(X_s)$ is a Lyndon structure. To construct a festoon in $X_s$, choose a bead of type $t$ with $t\mid s$, colour it in $b_t$ ways, repeat it $s/t$ times and rotate the festoon in one of $\rk(t)$ distinct ways. We have $\#Y_s=a_s$ by the relation between $(a_s)$ and $(b_s)$ in Theorem \ref{thm:abc}.
\end{proof}

\subsection{Examples}

We give a number of examples of Theorem \ref{thm:festoon-X} for $S=\Z_{\geq 1}$.

\begin{example}\label{ex:content}
Fix $\lambda\geq 1$. Set $c_1=\lambda$ and $c_n=0$ otherwise. Then $(X_n)$ is isomorphic to words of length $n$ on the alphabet $\{1,\dots,\lambda\}$ under rotation of letters. The associated Gauss congruence is $(\lambda^n)$, so $(X_n,\Cyc_n,[\exp_\lambda(n)]_q)$ satisfies the CSP for every $n\geq 1$.
\end{example}

\begin{example}
Set $c_1=c_2=1$ and $c_n=0$ otherwise. Then the values
\[a_n=\tr\begin{bmatrix}1 & 1 \\ 1 & 0\end{bmatrix}^n\]
are the Lucas numbers, satisfying $a_n=a_{n-1}+a_{n-2}$ for $n>2$. By Theorem \ref{thm:abc-q} an associated $q$-Gauss congruence is
\[f_n(q)=\sum_{k=0}^{\lfloor n/2\rfloor}\frac{[n]_q}{[n-k]_q}\qbinom{n-k}{k}.\]
If we label the beads as
\[\begin{tikzpicture}[scale = 0.8]
\draw (0,0) circle (0.3);
\node at (0,0) {$0$};
\draw ($(2,0)+(-0.3,-0.3)$) arc (270:90:0.3) -- ($(2,0)+(0.3,0.3)$) arc (90:-90:0.3) -- cycle;
\node at (2,0) {$0\,\,\,\,1$};
\end{tikzpicture}\]
we see that $(X_n)$ is isomorphic to the Lyndon structure of cyclic binary words of length $n$ where the 1s are never adjacent. For example:
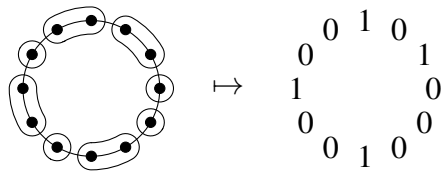
\[\begin{tikzpicture}[scale = 0.9]
\draw \tube{-2}{0}{1}{0.2}{0}{0};
\draw \tube{-2}{0}{1}{0.2}{30}{60};
\draw \tube{-2}{0}{1}{0.2}{90}{120};
\draw \tube{-2}{0}{1}{0.2}{150}{150};
\draw \tube{-2}{0}{1}{0.2}{180}{210};
\draw \tube{-2}{0}{1}{0.2}{240}{240};
\draw \tube{-2}{0}{1}{0.2}{270}{300};
\draw \tube{-2}{0}{1}{0.2}{330}{330};

\node at ($(2,0)+(0:1)$) {0};
\node at ($(2,0)+(30:1)$) {1};
\node at ($(2,0)+(60:1)$) {0};
\node at ($(2,0)+(90:1)$) {1};
\node at ($(2,0)+(120:1)$) {0};
\node at ($(2,0)+(150:1)$) {0};
\node at ($(2,0)+(180:1)$) {1};
\node at ($(2,0)+(210:1)$) {0};
\node at ($(2,0)+(240:1)$) {0};
\node at ($(2,0)+(270:1)$) {1};
\node at ($(2,0)+(300:1)$) {0};
\node at ($(2,0)+(330:1)$) {0};

\node at (0,0) {$\mapsto$};

\draw (-2,0) circle (1);
\foreach \ang in {0,1,2,3,4,5,6,7,8,9,10,11} {
	\draw [fill = black] ($(-2,0)+(30*\ang:1)$) circle (0.075);
}

\end{tikzpicture}\]
The triple $(X_n,\Cyc_n,f_n(q))$ then satisfies the CSP for every $n\geq 1$.
\end{example}
\begin{remark}
The above example is from \cite{q-congruences}, where the alternate $q$-Gauss congruence
\[g_n(q)=\tr\left(\begin{bmatrix}1 & 1 \\ 1 & 0\end{bmatrix}\begin{bmatrix}1 & q \\ 1 & 0\end{bmatrix}\cdots \begin{bmatrix}1 & q^{n-1} \\ 1 & 0\end{bmatrix}\right)\]
is used. As usual, this implies $f_n(q)\equiv g_n(q)\pmod{q^n-1}$ for every $n\geq 1$.
\end{remark}

\begin{example}
Set $c_2=3$, $c_3=2$ and $c_n=0$ otherwise. Then $a_n=\tr(M^n)$, where
\[M=\begin{bmatrix}0 & 3 & 2 \\ 1 & 0 & 0 \\ 0 & 1 & 0\end{bmatrix}\sim\begin{bmatrix}2 & 0 & 0 \\ 0 & -1 & 1 \\ 0 & 0 & -1\end{bmatrix}.\]
So $a_n=2^n+2(-1)^n$ and $f_n(q)=[\exp_2(n)]_q+2\eps_n(q)$ gives an associated $q$-Gauss congruence. We claim that $(X_n)$ is isomorphic to the Lyndon structure of cyclic words of length $n$ on the alphabet $\{0,1,2\}$ with no adjacent pairs. To see this, label the beads as
\[\begin{tikzpicture}[scale = 0.8]
\draw ($(0,0)+(-0.3,-0.3)$) arc (270:90:0.3) -- ($(0,0)+(0.3,0.3)$) arc (90:-90:0.3) -- cycle;
\node at (0,0) {$0\,\,\,\,1$};

\draw ($(2,0)+(-0.3,-0.3)$) arc (270:90:0.3) -- ($(2,0)+(0.3,0.3)$) arc (90:-90:0.3) -- cycle;
\node at (2,0) {$0\,\,\,\,2$};

\draw ($(4,0)+(-0.3,-0.3)$) arc (270:90:0.3) -- ($(4,0)+(0.3,0.3)$) arc (90:-90:0.3) -- cycle;
\node at (4,0) {$\text{x}\,\,\,\,\text{x}$};

\draw ($(6.2,0)+(-0.55,-0.3)$) arc (270:90:0.3) -- ($(6.2,0)+(0.55,0.3)$) arc (90:-90:0.3) -- cycle;
\node at (6.2,0) {$0\,\,\,\,1\,\,\,\,2$};

\draw ($(8.6,0)+(-0.55,-0.3)$) arc (270:90:0.3) -- ($(8.6,0)+(0.55,0.3)$) arc (90:-90:0.3) -- cycle;
\node at (8.6,0) {$0\,\,\,\,2\,\,\,\,1$};
\end{tikzpicture}\]
where the bead middle bead is `12' by default, but becomes `21' if it follows a bead ending in 1. Again, $(X_n,\Cyc_n,f_n(q))$ satisfies the CSP for every $n\geq 1$.
\end{example}

\subsection{Counting by number of beads}

We refine the enumeration of festoons coloured by $(c_s)$ by keeping track of the number of beads. Given a ranked semigroup $S$ and an integer sequence $(c_s)$, define the ranked semigroup $\hat S:=S[\Z_{\geq 1}]$ and the integer sequence $(\hat c_{s,k})$ on $\hat S$ by
\[\hat c_{s,k}=\begin{cases} c_s & \text{if }k = 1,\\
0 & \text{otherwise}.\end{cases}\]
If $(X_{s,k})$ is the Lyndon structure of festoons coloured by $(\hat c_{s,k})$, then $X_{s,k}$ is the set of festoons of type $s$ coloured by $(c_s)$ with exactly $k$ beads. In particular,
\[\#X_{s,k}=\sum_{s_1+\cdots+s_k=s}\rk(s_1)c_{s_1}\cdots c_{s_k}.\]
When $S=\Z_{\geq 1}$, the theory of Riordan arrays gives a surprising alternate enumeration for $(X_{n,k})$, the Lyndon structure of festoons of length $n$ with $k$ beads coloured by $(c_n)$.

\begin{theorem}\label{thm:riordan}
If $C(x)=\sum_{n\geq 1}c_nx^n$ satisfies the functional equation $C(x)=xD(C(x))$ for some formal Laurent series $D(t)\in\Z((t))$, then
\[\#X_{n,k}=[t^{n-k}]D(t)^n.\]
\end{theorem}

We give two Catalan-based examples below, and a Schr\"oder-based example in Example \ref{ex:strict-schroder}. Let $\mathrm{Cat}_n$ denote the $n^\text{th}$ Catalan number.

\begin{example}
If $D(t)=1+t^2$, then solving for $C(x)$ gives
\[c_n=\begin{cases}\mathrm{Cat}_{(n-1)/2} & \text{if $n$ is odd},\\
0 & \text{otherwise}.\end{cases}\] 
So $c_n$ counts words of length $n-1$ on the alphabet $\{(,)\}$ which have matched brackets. Adding a new letter $|$ to the end of each word, we label each of the $c_n$ beads of length $n$ by such a word of length $n$. 

Using this labelling, $(X_{n,k})$ is isomorphic to the Lyndon structure of cyclic words on the alphabet $\{(,),|\}$ with $n$ letters, $k$ of which are $|$, where brackets are matched and no letter $|$ occurs within a bracket pair. For example, $X_{6,2}$ has orbits:
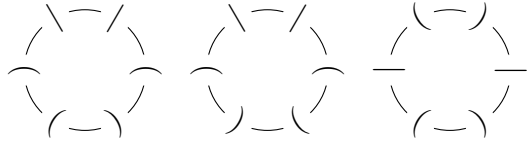
\[\begin{tikzpicture}[scale = 0.8]
\foreach \x in {0,1,2} {
	\draw (3*\x,0) circle (1);
	\foreach \ang in {0,1,2,3,4,5} {
		\draw [white, fill = white] ($(3*\x,0)+(60*\ang:1)$) circle (0.25);
	}
}
\node [rotate = -90] at ($(0,0)+(0:1)$) {$($};
\node [rotate = -30] at ($(0,0)+(60:1)$) {$|$};
\node [rotate = 30] at ($(0,0)+(120:1)$) {$|$};
\node [rotate = 90] at ($(0,0)+(180:1)$) {$)$};
\node [rotate = 150] at ($(0,0)+(240:1)$) {$)$};
\node [rotate = 210] at ($(0,0)+(300:1)$) {$($};

\node [rotate = -90] at ($(3,0)+(0:1)$) {$($};
\node [rotate = -30] at ($(3,0)+(60:1)$) {$|$};
\node [rotate = 30] at ($(3,0)+(120:1)$) {$|$};
\node [rotate = 90] at ($(3,0)+(180:1)$) {$)$};
\node [rotate = 150] at ($(3,0)+(240:1)$) {$($};
\node [rotate = 210] at ($(3,0)+(300:1)$) {$)$};

\node [rotate = -90] at ($(6,0)+(0:1)$) {$|$};
\node [rotate = -30] at ($(6,0)+(60:1)$) {$)$};
\node [rotate = 30] at ($(6,0)+(120:1)$) {$($};
\node [rotate = 90] at ($(6,0)+(180:1)$) {$|$};
\node [rotate = 150] at ($(6,0)+(240:1)$) {$)$};
\node [rotate = 210] at ($(6,0)+(300:1)$) {$($};
\end{tikzpicture}\]
By Theorem \ref{thm:riordan}, $\#X_{n,k}$ is equal to the number of compositions of $n-k$ into $n$ parts from the alphabet $\{0,2\}$. Taking the $q$-analogue of this enumeration gives the $q$-Gauss congruence
\[f_{n,k}(q)=\begin{cases}\text{\small $\qbinom{n}{(n-k)/2}$} & \text{if $n-k$ is even},\\0 & \text{otherwise}\end{cases}\]
and hence $(X_{n,k},\Cyc_n,f_{n,k}(q))$ satisifes the CSP for every $n,k\geq 1$.
\end{example}

\begin{example}
If $D(t)=1/(1-t)$ then $c_n=\mathrm{Cat}_{n-1}$ for every $n\geq 1$. By ($\text{j}^5$) of Stanley's Catalan addendum \cite{catalan}, $c_n$ counts ways to fully connect the vertices of the interval graph $\mathcal{I}_n$ by noncrossing arcs above the interval with distinct right endpoints. Labelling the beads by such pictures, $(X_{n,k})$ is isomorphic to the following Lyndon structure: colour $k$ vertices of $\Gamma_n$,  and add noncrossing arcs between the vertices so that from every vertex there is a unique path along arcs travelling clockwise, and this path terminates at the nearest coloured vertex. We exhibit some elements of $X_{8,3}$ below.
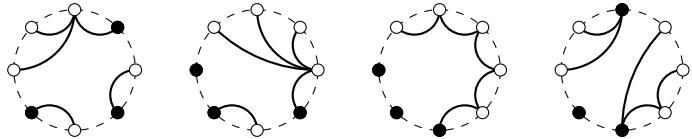
\[\begin{tikzpicture}[x=-1cm, scale = 0.8]
\draw [thick] \ortharc{0}{0}{1}{0}{90};
\draw [thick] \ortharc{0}{0}{1}{45}{90};
\draw [thick] \ortharc{0}{0}{1}{135}{270};
\draw [thick] \ortharc{0}{0}{1}{180}{225};
\draw [thick] \ortharc{0}{0}{1}{225}{270};

\draw [thick] \ortharc{3}{0}{1}{45}{90};
\draw [thick] \ortharc{3}{0}{1}{90}{135};
\draw [thick] \ortharc{3}{0}{1}{135}{180};
\draw [thick] \ortharc{3}{0}{1}{180}{225};
\draw [thick] \ortharc{3}{0}{1}{225}{270};

\draw [thick] \ortharc{6}{0}{1}{45}{180};
\draw [thick] \ortharc{6}{0}{1}{90}{180};
\draw [thick] \ortharc{6}{0}{1}{135}{180};
\draw [thick] \ortharc{6}{0}{1}{180}{225};
\draw [thick] \ortharc{6}{0}{1}{270}{315};

\draw [thick] \ortharc{9}{0}{1}{0}{90};
\draw [thick] \ortharc{9}{0}{1}{45}{90};
\draw [thick] \ortharc{9}{0}{1}{90}{135};
\draw [thick] \ortharc{9}{0}{1}{180}{225};
\draw [thick] \ortharc{9}{0}{1}{270}{315};

\foreach \x in {0,1,2,3} {
	\draw [dashed, thin] (3*\x,0) circle (1);
	\foreach \ang in {0,1,2,3,4,5,6,7} {
		\draw [fill = white] ($(3*\x,0)+(45*\ang:1)$) circle (0.1);
	}
}

\draw [fill = black] ($(0,0)+(90:1)$) circle (0.1);
\draw [fill = black] ($(0,0)+(270:1)$) circle (0.1);
\draw [fill = black] ($(0,0)+(315:1)$) circle (0.1);
\draw [fill = black] ($(3,0)+(270:1)$) circle (0.1);
\draw [fill = black] ($(3,0)+(315:1)$) circle (0.1);
\draw [fill = black] ($(3,0)+(360:1)$) circle (0.1);
\draw [fill = black] ($(6,0)+(225:1)$) circle (0.1);
\draw [fill = black] ($(6,0)+(315:1)$) circle (0.1);
\draw [fill = black] ($(6,0)+(360:1)$) circle (0.1);
\draw [fill = black] ($(9,0)+(135:1)$) circle (0.1);
\draw [fill = black] ($(9,0)+(225:1)$) circle (0.1);
\draw [fill = black] ($(9,0)+(315:1)$) circle (0.1);
\end{tikzpicture}\]
By Theorem \ref{thm:riordan}, $\#X_{n,k}$ is the number of compositions of $n-k$ into $n$ parts from the alphabet $\{0,1,2,\dots\}$. An associated $q$-Gauss congruence is
\[f_{n,k}(q)=\qbinom{2n-k-1}{n-k}\]
and $(X_{n,k},\Cyc_n,f_{n,k}(q))$ satisfies the CSP for every $n,k\geq 1$.
\end{example}

\section{Tubings} \label{sect:tubings}

We give two examples of Lyndon structures and their cyclic sieving results for tubings on the cycle graph $\Gamma_n$. The first uses Theorem \ref{thm:riordan}. We begin with some background on these objects.

\subsection{Background}

Tubings are a combinatorial object on graphs which correspond to the faces of the graph associahedron (see \cite{postnikov}). In this section we consider tubings on the interval and cycle graphs, whose graph associahedra correspond to the associahedron and cyclohedron respectively.

\begin{definition}\label{def:tubing}
Let $\Gamma$ be a connected simple graph. A \emph{tube} is a nonempty connected subset of vertices in $\Gamma$, and a \emph{tubing} is a set $T$ of tubes in $\Gamma$ so that for any $t_1,t_2\in T$:
\begin{itemize}
\item $t_1\subseteq t_2$, $t_2\subseteq t_1$, or
\item if $v_1$ and $v_2$ are vertices in $t_1$ and $t_2$ respectively, then $v_1\neq v_2$ and there is no edge between $v_1$ and $v_2$ in $\Gamma$.
\end{itemize}
Given a tubing $T$, a vertex in $\Gamma$ is \emph{free} if it is not contained in any tube of $T$. A tubing is \emph{proper} if there are no free vertices.
\end{definition}

On the interval graph $\mathcal{I}_n$, a vertex in $\mathcal{I}_n$ is \emph{final} if it is the last vertex in a tube which is not contained in a subtube. Below is the set of proper tubings on $\mathcal{I}_3$, with final vertices coloured in black.
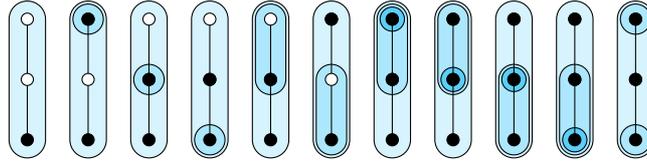
\[\begin{tikzpicture}[scale = 0.8]
\foreach \x in {0,1,2,3,4,5,6,7,8,9,10} {
	\draw [fill = blu!25] (\x-0.3,2) arc (180:0:0.3) -- (\x+0.3,0) arc (0:-180:0.3) -- cycle;
}
\foreach \x in {4,6,7} {\draw [fill = blu!50] (\x-0.25,2) arc (180:0:0.25) -- (\x+0.25,1) arc (0:-180:0.25) -- cycle;}
\foreach \x in {5,8,9} {\draw [fill = blu!50] (\x-0.25,1) arc (180:0:0.25) -- (\x+0.25,0) arc (0:-180:0.25) -- cycle;}
\foreach \x in {1,10} {\draw [fill = blu!50] (\x-0.25,2) arc (180:0:0.25) -- (\x+0.25,2) arc (0:-180:0.25) -- cycle;}
\foreach \x in {6} {\draw [fill = blu] (\x-0.2,2) arc (180:0:0.2) -- (\x+0.2,2) arc (0:-180:0.2) -- cycle;}
\foreach \x in {2} {\draw [fill = blu!50] (\x-0.25,1) arc (180:0:0.25) -- (\x+0.25,1) arc (0:-180:0.25) -- cycle;}
\foreach \x in {7,8} {\draw [fill = blu] (\x-0.2,1) arc (180:0:0.2) -- (\x+0.2,1) arc (0:-180:0.2) -- cycle;}
\foreach \x in {3,10} {\draw [fill = blu!50] (\x-0.25,0) arc (180:0:0.25) -- (\x+0.25,0) arc (0:-180:0.25) -- cycle;}
\foreach \x in {9} {\draw [fill = blu] (\x-0.2,0) arc (180:0:0.2) -- (\x+0.2,0) arc (0:-180:0.2) -- cycle;}

\foreach \x in {0,1,2,3,4,5,6,7,8,9,10} {
	\draw (\x,0) -- (\x,2);
	\foreach \y in {1,2} {\draw [fill = white] (\x,\y) circle (0.1);}
	\draw [fill = black] (\x,0) circle (0.1);
}
\draw [fill = black] (1,2) circle (0.1);
\draw [fill = black] (2,1) circle (0.1);
\draw [fill = black] (3,1) circle (0.1);
\draw [fill = black] (4,1) circle (0.1);
\draw [fill = black] (5,2) circle (0.1);
\foreach \x in {6,7,8,9,10} {
	\draw [fill = black] (\x,1) circle (0.1);
	\draw [fill = black] (\x,2) circle (0.1);
}
\end{tikzpicture}\]

\subsection{Bijections with lattice paths}

We count tubings via bijections with lattice path walks. Fix the vectors 
\[{\nearrow}:=(1,1),\quad{\searrow}:=(1,-1),\quad{\longrightarrow}:=(2,0)\]
in $\Z^2$, and let $V:=\{{\nearrow},{\searrow},{\longrightarrow}\}$. A \emph{Delannoy path} of \emph{length $n$} is a word $v_1\cdots v_m$ in $V$ such that $v_1+\cdots+v_m=(n,0)$. Each $v_i$ is called a \emph{step} and the \emph{height} of this step is the $y$-coordinate of $v_1+\cdots+v_{i-1}$. A \emph{Schr\"oder path} is a Delannoy path where all steps have nonnegative height, and this path is \emph{strict} if all steps ${\longrightarrow}$ have positive height.

We use these lattice path walks to enumerate tubings by finding relevant bijections. In the case of $\Gamma_n$, we believe such a bijection is novel. We also give a lemma about strict Schr\"oder paths which will be used in Example \ref{ex:strict-schroder}.

\begin{lemma}\label{lem:schroder-bij}
There is a bijection between tubings on $\mathcal{I}_n$ and Schr\"oder paths of length $2n$ where:
\begin{itemize}
\item tubes correspond to steps ${\nearrow}$,
\item nonfinal vertices correspond to steps ${\longrightarrow}$, and
\item free vertices correspond to steps ${\longrightarrow}$ at height $0$.
\end{itemize}
\end{lemma}

\begin{lemma}\label{lem:cycle-tubings}
There is a bijection between improper tubings on $\Gamma_{n}$ with $k$ tubes and Delannoy paths of length $2(n-1)$ with $n-k-1$ steps $\longrightarrow$.
\end{lemma}

\begin{lemma}\label{lem:strict-schroder}
Let $c_n$ be the number of strict Schr\"oder paths of length $2(n-1)$, and define the generating function $C(x)=\sum_{n\geq 1}c_nx^n$. Then
\[C(x)=x\left(1+\frac{1-C(x)}{1-2C(x)}\right)\]
\end{lemma}

\subsection{Cyclic sieving results}

Let $c_n$ be the number of tubings of $\mathcal{I}_n$ where only the last vertex is free, which is equal to the number of strict Schr\"oder paths of length $n-1$. Colouring beads by such tubings, we observe that improper tubings on $\Gamma_n$ give a Lyndon structure isomorphic to festoons coloured by $(c_n)$. In the inverse map, a new bead is started after every free vertex.

\begin{example}\label{ex:strict-schroder}
Let $X_{n,k}$ be the set of improper tubings on $\Gamma_n$ with $k$ free vertices. Then $(X_{n,k})$ is a Lyndon structure isomorphic to festoons of length $n$ with $k$ beads coloured by $(c_n)$. Below are some elements of $X_{8,2}$.
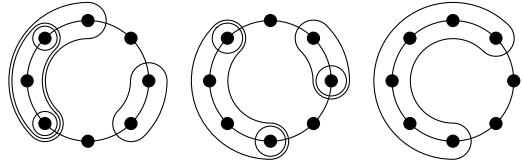
\[\begin{tikzpicture}[scale = 0.8]
\draw \tube{0}{0}{1}{0.3}{2*45}{5*45};
\draw \tube{0}{0}{1}{0.25}{3*45}{5*45};
\draw \tube{0}{0}{1}{0.2}{3*45}{3*45};
\draw \tube{0}{0}{1}{0.2}{5*45}{5*45};
\draw \tube{0}{0}{1}{0.3}{7*45}{8*45};

\draw \tube{3}{0}{1}{0.3}{3*45}{6*45};
\draw \tube{3}{0}{1}{0.25}{6*45}{6*45};
\draw \tube{3}{0}{1}{0.25}{3*45}{3*45};
\draw \tube{3}{0}{1}{0.3}{0*45}{1*45};
\draw \tube{3}{0}{1}{0.25}{0*45}{0*45};

\draw \tube{6}{0}{1}{0.3}{1*45}{6*45};

\foreach \x in {0,1,2} {
	\draw (3*\x,0) circle (1);
}
\foreach \x in {0,1,2} {
	\foreach \ang in {0,1,2,3,4,5,6,7} {
		\draw [fill = black] ($(3*\x,0)+(45*\ang:1)$) circle (0.1);
	}
}
\end{tikzpicture}\]
By Lemma \ref{lem:strict-schroder} and Theorem \ref{thm:riordan}, $\#X_{n,k}=[t^{n-k}]D(t)^n$, where
\[D(t)=1+\frac{1-t}{1-2t}=1+t+2t^2+4t^3+\cdots.\] 
So $(X_{n,k})$ is enumerated by compositions of $n-k$ into $n$ parts from the alphabet $\{0,1,2,\dots\}$, where parts of size $i\geq 1$ are coloured in one of $2^{i-1}$ ways. We claim that
\[f_{n,k}(q)=\sum_{m=0}^{n-k-1}\qbinom{n}{m+k}\qbinom{n-k-1}{m}[\exp_2(m)]_q.\]
gives an associated $q$-Gauss congruence $(f_{n,k}(q))$. Letting $m+k$ be the number of nonzero parts in each composition, we see that $\#X_{n,k}=f_{n,k}(1)$. By chaining, the family
\[g_{n,k,m}(q)=\qbinom{n}{k}\qbinom{k+m-1}{m}[\exp_2(m)]_q\]
satisfies $q$-Gauss congruence on $\Z_{\geq 1}[\Z][\Z]$. Pulling back along $(n,k,m)\mapsto (n,n-m-k,m)$ and pushing forward along $(n,k,m)\mapsto (n,k)$, we obtain $(f_{n,k}(q))$. Hence, $(X_{n,k},\Cyc_n,f_{n,k}(q))$ satisfies the CSP for every $n,k\geq 1$.
\end{example}

\begin{example}
Let $X_{n,k}$ be the set of improper tubings on $\Gamma_n$ with $k$ tubes. Then $(X_{n,k})$ is a Lyndon structure on $\Z_{\geq 1}[\Z_{\geq 0}]$ and by Lemma \ref{lem:cycle-tubings} the polynomials
\[f_{n,k}(q)=\qbinom{n+k-1}{k}\qbinom{n-1}{k}\]
satisfy $f_{n,k}(1)=\#X_{n,k}$. Moreover, $(f_{n,k}(q))$ satisfies $q$-Gauss congruence, obtained by pulling back the $q$-Gauss congruence
\[g_{n,m,k}(q)=\qbinom{n+m-1}{m}\qbinom{m+k-1}{m}\]
on $\Z_{\geq 1}[\Z][\Z]$ along the morphism $(n,k)\mapsto (n,k,n-k)$. Hence, $(X_{n,k},\Cyc_n,f_{n,k}(q))$ satisfies the CSP for every $n\geq 1$ and $k\geq 0$. Taking the sum over all $k$ recovers Theorem \ref{thm:tubings-intro}.
\end{example}
\begin{remark}
Fix $\lambda\geq 1$. Multiplying $f_{n,k}(q)$ by $[\exp_\lambda(k)]_q$ gives a generalisation of the Theorem where each tube can be coloured in one of $\lambda$ ways.
\end{remark}

\section{Signed festoons} \label{sect:signed}

Lyndon structures can only be combinatorial analogues for Gauss congruences with $b_s\geq 0$ for every $s$. By taking a signed enumeration of festoons coloured by an arbitrary integer sequence $(c_s)$, we state a weakening of Theorem \ref{thm:festoon-X} which holds for all Gauss congruences. We conclude with a motivating example. 

Fix a ranked semigroup $S$ with finite decomposition and an integer sequence $(c_s)$. Let $X_s$ be the set of festoons of type $s$ coloured by $(|c_s|)$. We call a festoon \emph{positive} if it has beads of types $s_1,s_2,\dots$ such that $c_{s_1}c_{s_2}\cdots>0$, and \emph{negative} otherwise. We call elements of $X_s$ \emph{signed festoons}, and let $X_s=X_s^+\sqcup X_s^-$ be the decomposition into positive and negative festoons. The signed enumeration of these festoons will be the Gauss congruence associated to $(c_s)$, and we obtain a signed analogue of the cyclic sieving phenomenon for odd ranks.

\begin{definition}
Let $X=X^+\sqcup X^-$ be a finite set with an action of $\Cyc_n$ such that $X^+$ and $X^-$ are fixed setwise. If $f(q)\in\Z[q]$ is a polynomial such that
\[f(\omega_d)=\#(X^+)^{\Cyc_d}-\#(X^-)^{\Cyc_d}\]
for every $d\mid n$, then $(X,\Cyc_n,f_n(q))$ satisfies the \emph{signed cyclic sieving phenomenon}.
\end{definition}

\begin{theorem}\label{thm:signed}
Fix an arbitrary sequence $(a_s)$ satisfying Gauss congruence on $S$, with associated sequence $(c_s)$ and $q$-Gauss congruence $(f_s(q))$. Let $X_s=X_s^+\sqcup X_s^-$ be the set of signed festoons of type $s$ coloured $(c_s)$. Then $a_s:=\#X_s^+-\#X_s^-$ for every $s$, and $(X_s,\Cyc_{\rk(s)},f_s(q))$ satisfies the signed cyclic sieving phenomenon when $\rk(s)$ is odd.
\end{theorem}

\begin{example}
Let $(a_n)$ be the Gauss congruence on $\Z_{\geq 1}$ given by $a_n=-\sigma(n)$. Then $b_n=-1$ for every $n\geq 1$ so by (\ref{eq:bc}):
\[c_n=-[x^n]\prod_{m\geq 1}\frac{1}{(1-x^m)}=-\#\{\text{partitions of $n$}\}.\]
The Lyndon structure $(X_n)$ of festoons coloured by $(|c_n|)$ is isomorphic to festoons of length $n$ where all beads have the same colour, and barriers can be drawn between the beads. We require there be at least one barrier, and that bead-length strictly increases when travelling clockwise only at barriers. For example, elements of $X_{12}$ include:
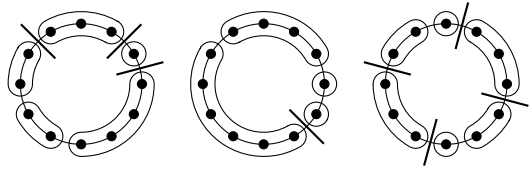
\[\begin{tikzpicture}[scale = 0.8]
\draw \tube{0}{0}{1}{0.2}{30*9}{30*12};
\draw \tube{0}{0}{1}{0.2}{30*7}{30*8};
\draw \tube{0}{0}{1}{0.2}{30*5}{30*6};
\draw \tube{0}{0}{1}{0.2}{30*2}{30*4};
\draw \tube{0}{0}{1}{0.2}{30*1}{30*1};

\draw [thick] (30*4.5:0.6) -- (30*4.5:1.4);
\draw [thick] (30*1.5:0.6) -- (30*1.5:1.4);
\draw [thick] (30*0.5:0.6) -- (30*0.5:1.4);

\draw \tube{3}{0}{1}{0.2}{30*5}{30*10};
\draw \tube{3}{0}{1}{0.2}{30*1}{30*4};
\draw \tube{3}{0}{1}{0.2}{30*0}{30*0};
\draw \tube{3}{0}{1}{0.2}{30*11}{30*11};

\draw [thick] ($(3,0)+(30*-1.5:0.6)$) -- ($(3,0)+(30*-1.5:1.4)$);

\draw \tube{6}{0}{1}{0.2}{30*10}{30*11};
\draw \tube{6}{0}{1}{0.2}{30*9}{30*9};
\draw \tube{6}{0}{1}{0.2}{30*6}{30*8};
\draw \tube{6}{0}{1}{0.2}{30*4}{30*5};
\draw \tube{6}{0}{1}{0.2}{30*3}{30*3};
\draw \tube{6}{0}{1}{0.2}{30*0}{30*2};

\draw [thick] ($(6,0)+(30*11.5:0.6)$) -- ($(6,0)+(30*11.5:1.4)$);
\draw [thick] ($(6,0)+(30*8.5:0.6)$) -- ($(6,0)+(30*8.5:1.4)$);
\draw [thick] ($(6,0)+(30*5.5:0.6)$) -- ($(6,0)+(30*5.5:1.4)$);
\draw [thick] ($(6,0)+(30*2.5:0.6)$) -- ($(6,0)+(30*2.5:1.4)$);

\foreach \x in {0,3,6} {
	\draw (\x,0) circle (1);
	\foreach \ang in {0,1,2,3,4,5,6,7,8,9,10,11} {
		\draw [fill = black] ($(\x,0)+(30*\ang:1)$) circle (0.075);
	}
}
\end{tikzpicture}\]
Each element is positive or negative depending on whether the number of barriers is even or odd, and we have $\#X_n^+-\#X_n^-=-\sigma(n)$. Let $Y_n$ be the set of such drawings where any number of barriers is allowed (including zero). Then $\#Y_n^+=\#X_n^++\sigma(n)$, so $\#Y_n^+-\#Y_n^-=0$ for every $n\geq 1$ and $(Y_n,\Cyc_n,0)$ satisfies the signed cyclic sieving phenomenon when $n$ is odd.
\end{example}

\section{Proofs}\label{sect:proofs}

\begin{proof}[Proof of Theorem \ref{thm:abc}]
We prove the equivalence between the first two statements by translating the more general result of \cite[Theorem 16]{dold-poset} to our language. Write $t\wr s$ if $t\mid s$ but $t\neq s$. By induction on $\rk$, there is a unique function $\mu^*:S\times S\to\Z$ satisfying
\[\mu^*(t,s)=\begin{cases}1&\text{if }t=s,\\
-\sum_{t\mid u\wr s}\mu^*(t,u)&\text{if }t\wr s,\\
0&\text{otherwise.}\end{cases}\]
This is the \emph{generalised M\"obius function} on the poset of $S$ under division. By generalised M\"obius inversion (see \cite[Section 3.7]{ec1}) we have
\[a_s=\sum_{t\mid s}\rk(t)b_t\quad\text{if and only if}\quad\rk(s)b_s=\sum_{t\mid s}\mu^*(t,s)a_t.\]
For $t\mid s$, the interval
\[[t,s]:=\{u\in S\mid t\mid u\mid s\}\]
is isomorphic to the poset of integers $\{1,\dots,s/t\}$ under division, so $\mu^*(t,s)=\mu(s/t)$ and
\[a_s=\sum_{t\mid s}\rk(t)b_t\quad\text{if and only if}\quad b_s=\frac{1}{\rk(s)}\sum_{t\mid s}\mu(s/t)a_t.\]
There is always a unique rational sequence of values $(b_s)$ satisfying the appropriate relationship with $(a_s)$, and the above equivalence proves that $(a_s)$ satisfies Gauss congruence if and only if $(b_s)$ is an integer sequence.

We give a novel proof involving $q$-Gauss congruence of the equivalence between the first and third statements. Recall from (\ref{eq:ac2S}) that for any integer sequence $(c_s)$ we have the identity
\[a_s:=\sum_{s_1+s_2+\cdots=s}\rk(s_1)c_{s_1}c_{s_2}\cdots=\sum_{\substack{\alpha\in S^*\\\tau(\alpha)=s}}\frac{\rk(s)}{|\alpha|}\binom{|\alpha|}{\alpha}\prod_{t\in S} c_t^{\alpha(t)}\]
In proving Theorem \ref{thm:abc-q} we (independently) construct a $q$-Gauss congruence $(g_s(q))$ evaluating at $q=1$ to $(a_s)$, proving that this sequence satisfies Gauss congruence. Conversely, if $(a_s)$ is any integer sequence satisfying Gauss congruence, we can find a rational sequence $(c_s)$ satisfying the equation with $(a_s)$. Define the sequence $(c_s')$ by $c_s':=\lfloor c_s\rfloor\in\Z$. Then the sequence $(a_s')$ defined by
\[a_s'=\sum_{s_1+s_2+\cdots=s}\rk(s_1)c_{s_1}'c_{s_2}'\cdots\]
satisfies Gauss congruence. Assume by contradiction that $(c_s)$ is not an integer sequence, and let $s$ be of minimal rank such that $c_s\not\in\Z$. Then $a_t=a_t'$ for all $t<s$, and $a_s-a_s'=\rk(s)(c_s-c_s')$, so
\[c_s=c_s'+\frac{1}{\rk(s)}\sum_{t\mid s}\mu(s/t)a_t-\frac{1}{\rk(s)}\sum_{t\mid s}\mu(s/t)a_t'.\]
But each term on the right-hand side is an integer, which contradicts $c_s\not\in\Z$.
\end{proof}

\begin{proof}[Proof of Proposition \ref{prop:root-sub}]
We first prove that if $(f_s(q))$ satisfies $q$-Gauss congruence and $(g_s(q))$ is a sequence of integer polynomials satisfying $f_s(1)=g_s(1)$, then $(g_s(q))$ satisfies $q$-Gauss congruence if and only if 
\[f_s(q)\equiv g_s(q)\pmod{q^{\rk(s)}-1}\]
for every $s\in S$.

Set $h_s(q):=f_s(q)-g_s(q)$. Then $h_s(1)=0$ for every $s\in S$. We prove that $h_s(q)$ satisfies $q$-Gauss congruence if and only if $h_s(q)\equiv 0\pmod{q^{\rk(s)}-1}$ for every $s\in S$ by induction on $\rk$. Consider the equation
\[\sum_{t\mid s}\mu(s/t)h_t(q^{s/t}).\]
By induction we assume that $h_t(q)\equiv 0\pmod{q^{\rk(t)}-1}$ for $t\mid s$ with $t\neq s$, and hence $h_t(q^{s/t})\equiv 0\pmod{q^{\rk(s)}-1}$. The remaining term is $h_s(q)$. Since $h_s(1)=0$, we have $h_s(q)\equiv 0\pmod{[\rk(s)]_q}$ if and only if $h_s(q)\equiv 0\pmod{q^{\rk(s)}-1}$ as required.

We now prove the Proposition, and without loss of generality we can assume by reducing $f_s(q)$ mod $(q^{\rk(s)}-1)$ that $\deg(f_s)<\rk(s)$ for every $s\in S$. By induction, find rational values $b_s$ such that
\[f_s(1)=\sum_{t\mid s}\rk(t)b_t\]
for every $s\in S$. Consider the sequence $(g_s(q))$ given by
\[g_s(q):=\sum_{t\mid s}[\rk(t)]_{q^{s/t}}b_t.\]
If $q=\omega_d$ for $d\mid \rk(s)$, then $q^{s/t}$ is a $\rk(t)^\text{th}$ root of unity which is equal to $1$ if and only if $d\mid s/t$, or equivalently, $t\mid u$ for some $u\in s/d$. Then
\[g_s(\omega_d)=\sum_{u\in s/d}\sum_{t\mid u}\rk(t)b_t=\sum_{u\in s/d}f_u(1).\]
By the degree restrictions, $f_s(q)$ has the appropriate root of unity evaluations if and only if $f_s(q)=g_s(q)$, and since each $f_s(q)$ is an integer polynomial we obtain by induction that $b_s\in\Z$. Conversely, if $f_s(q)=g_s(q)$ for an arbitrary integer sequence $(b_s)$, then $(f_s(q))$ satisfies $q$-Gauss congruence since for every $s\in S$ we can use M\"obius inversion of polynomials to obtain
\[\sum_{t\mid s}\mu(s/t)f_t(q^{s/t})=b_s[\rk(s)]_q\equiv 0\pmod{[\rk(s)]_q}.\]
\end{proof}

\begin{proof}[Proof of Lemma \ref{lem:pushforward}]
We compute $g_t(\omega_d)$ for every $d\mid\rk(t)$ (implying $d\mid\rk(s)$):
\[g_t(\omega_d)=\sum_{s\in\phi^{-1}(t)}f_s(\omega_d)=\sum_{s\in \phi^{-1}(t)}\sum_{u\in s/d}f_u(1).\]
Every $u\in S$ with $\phi(u)\in t/d$ occurs in this sum exactly once, so
\[g_t(\omega_d)=\sum_{u\in \phi^{-1}(t/d)}f_u(1)=\sum_{u\in t/d}g_u(1)\]
as required.
\end{proof}

\begin{proof}[Proof of Lemma \ref{lem:pullback}]
For $d\mid\rk(s)\mid\rk(\phi(s))$, we have
\[f_s(\omega_d)=g_{\phi(s)}(\omega_d)=\sum_{t\in \phi(s)/d}g_t(1)=\sum_{u\in s/d}g_{\phi(u)}(1)=\sum_{u\in s/d}f_u(1).\]
by the property on $\phi$.
\end{proof}

\begin{proof}[Proof of Lemma \ref{lem:multiply}]
For any $d\mid\rk(s)$,
\[f_s(\omega_d)g_s(\omega_d)=\sum_{t\in s/d}f_t(1)\sum_{t\in s/d}g_t(1)=\sum_{t\in s/d}f_t(1)g_t(1)\]
where each sum has at most a single element since $S$ is torsion-free.
\end{proof}

\begin{proof}[Proof of Lemma \ref{lem:chaining-pre}]
We evaluate $h_{s,t,u}(\omega_d)$ for $d\mid \rk(s,t,u)=\rk(s)$. If $d\nmid (s,t,u)$, then $d\nmid (s,t)$ or $d\nmid(s,u)$, so $f_{s,t}(\omega_d)f_{s,u}(\omega_d)=0$. Otherwise,
\[h_{s,t,u}(\omega_d)=f_{(s,t)/d}(1)g_{(s,u)/d}(1)=f_{s/d,t/d}(1)g_{s/d,u/d}(1)=h_{(s,t,u)/d}(1).\]
\end{proof}

\begin{proof}[Proof of Lemma \ref{lem:chaining-suf}]
We evaluate $h_{s,t,u}(\omega_d)$ for $d\mid \rk(s,t,u)=\rk(s)$. If $d\nmid s$, or $d\mid s$ but $d\nmid t$, then $f_{s,t}(\omega_d)=0$. If $d\mid s$ and $d\mid t$ but $d\nmid u$, then $f_{t,u}(\omega_d)=0$. The final case is $d\mid (s,t,u)$, and we have
\[h_{s,t,u}(\omega_d)=f_{(s,t)/d}(1)g_{(s,u)/d}(1)=f_{s/d,t/d}(1)g_{s/d,u/d}(1)=h_{(s,t,u)/d}(1).\]
\end{proof}

\begin{proof}[Proof of Theorem \ref{thm:fund}]
We are required to show $f_\alpha$ is an integer polynomial satisfying the root-of-unity evaluations from Proposition \ref{prop:root-sub}. For $\alpha\in A^\psi$, set $\gcd(\alpha):=\gcd(\{\alpha(s)\mid s\in S\})$. Then
\[f_\alpha(q)=\frac{[\psi(\alpha)]_q}{[\gcd(\alpha)]_q}\cdot\frac{[\gcd(\alpha)]_q}{[|\alpha|]_q}\qbinom{|\alpha|}{\alpha}.\]
The first factor is an integer polynomial since $\gcd(\alpha)\mid\psi(\alpha)$, and the second is also an integer polynomial by \cite[Theorem 3.1]{qmulti-div}, so $f_\alpha(q)\in\Z[q]$. If $q=\omega_d$ for $d\mid\psi(\alpha)$, the first factor is zero unless $d\mid\gcd(\alpha)$. We recall \cite[Lemma 2.4]{sagan}: if $m\equiv n\pmod{d}$, then
\[\lim_{q\to\omega_d}\frac{[n]_q}{[m]_q}=\begin{cases}n/m & \text{if }m\equiv 0\pmod{d},\\
1 & \text{otherwise.}\end{cases}\]
Hence, for $d\mid\gcd(\alpha)$ we have $d\mid\alpha(s)$ for every $s\in S$ and $d\mid|\alpha|$, so
\begin{align*}\lim_{q\to\omega_d}f_\alpha(q)&=\frac{\psi(\alpha)}{|\alpha|}\frac{d\cdot 2d\cdots |\alpha|}{\prod_{s\in S}d\cdot 2d\cdots \alpha(s)}\\&=\frac{\psi(\alpha)/d}{|\alpha|/d}\frac{1\cdot 2\cdots |\alpha|/d}{\prod_{s\in S}1\cdot 2\cdots \alpha(s)/d}\\&=f_{\alpha/d}(1)\end{align*}
as required.
\end{proof}

\begin{proof}[Proof of Proposition \ref{prop:exp}]
The case $\lambda=0$ is clear. For $\lambda\geq 1$, let $A:=\{1,\dots,\lambda\}$. Elements of $A^*$ with rank $n$ are in correspondence with compositions of length $\lambda$ and sum $n$. Observe that
\[f_n(q)=\sum_{\substack{\alpha\in A^*\\ \rk(\alpha)=n}}f_\alpha(q),\]
gives a $q$-Gauss congruence, where $f_\alpha(q)$ is from Theorem \ref{thm:fund} (applied to $A^*$). We have $f_n(1)=\lambda^n$ since both sides count all words of length $n$ in $A$. When $\lambda<0$ we multiply each $f_n(q)$ by $\eps_n(q)$, and this preserves $q$-Gauss congruence since $\Z_{\geq 1}$ is torsion-free.
\end{proof}

\begin{proof}[Proof of Theorem \ref{thm:abc-q}]
The final statement follows from the proof of Proposition \ref{prop:root-sub}. It remains to prove that each equation for $(g_s(q))$ satisfies $q$-Gauss congruence and $g_s(1)=a_s$.

(a) The coefficients of $g_s(q)$ are integers by Proposition \ref{prop:gauss-other} applied to each Ramanujan's sum $\mu_j$. For every $d\mid \rk(s)$:
\begin{align*}
g_s(\omega_d)&=\frac{1}{\rk(s)}\sum_{j=0}^{\rk(s)-1}\left(\sum_{t\mid s}\mu_j(s/t)a_t\right)\omega_d^j\\
&=\frac{1}{\rk(s)}\sum_{t\mid s}a_t\sum_{j=0}^{\rk(s)-1}\omega_d^j\mu_j(s/t)\\
&=\frac{1}{\rk(s)}\sum_{t\mid s}\begin{cases}\rk(s)&\text{if }d=s/t\\
0 & \text{otherwise}\end{cases}\\
&=\sum_{t\in s/d}a_t.
\end{align*}
For $d=1$, this gives $g_s(1)=a_s$ for every $s\in S$, and hence $(g_s(q))$ satisfies $q$-Gauss congruence by Proposition \ref{prop:root-sub}.

(b) This case follows from the proof of Proposition \ref{prop:root-sub}.

(c) Assume $c_s\geq 0$ for every $s\in S$. Set $A:=\{(s,j)\mid n\geq 1,\,1\leq j\leq c_s\}$ and $\psi:A\to\Z$ by $\psi(s,j):=\rk(s)$. We also define $\psi$ on $S$ by $\psi(s)=\rk(s)$. The map $A\to S$ with $(s,j)\mapsto s$ induces a morphism $A^\psi\to S^\psi$ by $\Z$-linear extension. The pushfoward of $(f_\beta(q))$ for $\beta\in A^\psi$ from Theorem \ref{thm:fund} along the induced morphism gives the $q$-Gauss congruence
\[g_\alpha(q):=\sum_{\beta}\frac{[\psi(\beta)]_q}{[|\beta|]_q}\qbinom{[|\beta|]}{\beta}=\frac{[\psi(\alpha)]_q}{[|\alpha|]_q}\sum_{\beta}\qbinom{|\alpha|}{\beta}\]
on $S^\psi$, where the sum is over all $\beta\in A^\psi$ such that $\sum_j\beta(s,j)=\alpha(s)$ for every $s\in S$. Equivalently, $\beta(s):=\beta(s,1)\cdots\beta(s,c_s)$ is a composition with sum $\alpha(s)$. By factoring out the appropriate $q$-multinomial coefficient:
\begin{align*}g_\alpha(q)=&\frac{[\psi(\alpha)]_q}{[|\alpha|]_q}\qbinom{|\alpha|}{\alpha}\sum_{\beta}\prod_s\qbinom{\alpha(s)}{\beta(s)}\\
=&\frac{[\psi(\alpha)]_q}{[|\alpha|]_q}\qbinom{|\alpha|}{\alpha}\prod_s[\exp_{c_s}(\alpha(s))]_q\end{align*}
since $\beta(s)$ ranges over all compositions of length $c_s$ and sum $\alpha(s)$. Now pushfoward $(g_\alpha(q))$ along the canonical map $\tau:S^\psi\to S$ to get
\[f_s(q)=\sum_{\substack{\alpha\in S^*\\\tau(\alpha)=s}}\frac{[\rk(s)]_q}{[|\alpha|]_q}\qbinom{|\alpha|}{\alpha}\prod_{t}[\exp_{c_t}(\alpha(t))]_q\]
noting that $S^*$ and $S^\psi$ have the same elements in this case. This satisfies $f_s(1)=a_s$ by (\ref{eq:ac2S}).

In the general case, we construct $f_\beta(q)$ using the values $|c_s|$ for every $s\in S$, but replace $g_\alpha(q)$ with $g_\alpha(q)h_\alpha(q)$ in the above proof, where
\[h_\alpha(q)=\prod_{\substack{s\in S\\c_s<0}}\eps_{\alpha(s)}(q).\]
This replaces each $[\exp_{|c_s|}(\alpha(s))]_q$ with $[\exp_{c_s}(\alpha(s))]_q$. Multiplication preserves $q$-Gauss congruence since $S^\psi$ is torsion-free.
\end{proof}

\begin{proof}[Proof of Theorem \ref{thm:riordan}]
We observe that
\[\#X_{n,k}=\sum_{n_1+\cdots+n_k=n}n_1c_{n_1}\cdots c_{n_k}=xC'(x)C(x)^{k-1}\]
and hence the array $[r_{n,k}]$ defined by $r_{n,k}:=\#X_{n+1,k+1}$ is the Riordan array $\mathcal{R}(C'(x),C(x))$ (see \cite{riordan} for the general theory of these arrays). The proof now follows the `inverse-inverse trick' for Riordan arrays from \cite{invinv}.

Let $R=\mathcal{R}(d(t),h(t))$ be a proper Riordan array with values $r_{n,k}$. We combine the following two facts:
\begin{itemize} \item \cite[Theorem 3.2]{riordan} The matrix inverse $R^{-1}$ is $\mathcal{R}(1/d(\overline{h}(t)),\overline{h}(t))$, where $\overline{h}$ is the compositional inverse of $h$.
\item \cite[Theorem 2.2]{riordan-identities} If $R^{-1}=[r_{n,k}^*]$, then
\[r_{n,k}^*=[t^{n-k}]h'(t)/(d(t)(t/h(t))^{n+1}).\]
\end{itemize}
We apply the second fact to $R^{-1}$ from the first fact to obtain
\[r_{n,k}=[t^{n-k}]{\overline{h}}{}'(t)d(\overline{h}(t))(t/\overline{h}(t))^{n+1}.\]
In the case $\mathcal{R}(C'(x),C(x))$, this formula dramatically simplifies to
\[r_{n,k}=[t^{n-k}](t/\overline{C}(t))^{n+1}.\]
The equation $C(x)=xD(C(x))$ implies that $D(t)=t/\overline{C}(t)$, so $\#X_{n,k}=[t^{n-k}]D(t)^n$.
\end{proof}

\begin{proof}[Proof of Lemma \ref{lem:schroder-bij}]
Tubings on $\mathcal{I}_n$ can be written as a words in the alphabet $\{(,),\bullet\}$ where the brackets are matched, there are $n$ letters $\bullet$, and the word does not contain $)($ or a pair $(($ which is matched to a pair $))$.

\[\begin{tikzpicture}[scale = 0.8]
\draw (1,0) -- (6,0);
\foreach \x in {1,2,3,4,5,6} {
	\draw [fill = black] (\x,0) circle (0.1);
}
\draw (1,0.35) arc (90:270:0.35) -- (4,-0.35) arc (-90:90:0.35) -- cycle;
\draw (1,0.25) arc (90:270:0.25) -- (2,-0.25) arc (-90:90:0.25) -- cycle;
\draw (4,0.25) arc (90:270:0.25) -- (4,-0.25) arc (-90:90:0.25) -- cycle;
\draw (6,0.35) arc (90:270:0.35) -- (6,-0.35) arc (-90:90:0.35) -- cycle;

\node at (7,0) {$\leftrightarrow$};
\node at (10.175,0) {\Large $((\bullet\,\bullet)\bullet(\bullet))\bullet(\bullet)$};
\end{tikzpicture}\]
Given a tubing, construct a Schr\"oder path: when a tube begins add a step ${\nearrow}$, when a final vertex is passed add a step ${\searrow}$, and when a nonfinal vertex is passed add a step ${\longrightarrow}$.

Conversely, for a Schr\"oder path $v_1\cdots v_m$, we construct the tubing word. At each step ${\nearrow}$ add the letter $($, and for any other step add $\bullet$. For each step ${\nearrow}$ at height $h$, add the letter $)$ immediately before the next step ${\longrightarrow}$ or ${\searrow}$ at height $h$, or when the path ends. This gives the inverse map.
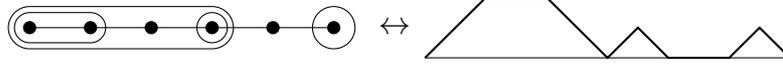
\[\begin{tikzpicture}[scale = 0.8]
\draw (1,0) -- (6,0);
\foreach \x in {1,2,3,4,5,6} {
	\draw [fill = black] (\x,0) circle (0.1);
}
\draw (1,0.35) arc (90:270:0.35) -- (4,-0.35) arc (-90:90:0.35) -- cycle;
\draw (1,0.25) arc (90:270:0.25) -- (2,-0.25) arc (-90:90:0.25) -- cycle;
\draw (4,0.25) arc (90:270:0.25) -- (4,-0.25) arc (-90:90:0.25) -- cycle;
\draw (6,0.35) arc (90:270:0.35) -- (6,-0.35) arc (-90:90:0.35) -- cycle;

\node at (7,0) {$\leftrightarrow$};
\draw [thick] (-0.5+8,-0.5) -- (-0.5+9,0.5) -- (-0.5+10,0.5) -- (-0.5+10.5,0) -- (-0.5+11,-0.5) -- (-0.5+11.5,0) -- (-0.5+12,-0.5) -- (-0.5+13,-0.5) -- (-0.5+13.5,0) -- (-0.5+14,-0.5);
\draw [thin] (-0.5+8,-0.5) -- (-0.5+14,-0.5);
\end{tikzpicture}\]
\end{proof}

\begin{proof}[Proof of Lemma \ref{lem:cycle-tubings}]
We prove that both sets biject with pairs $(p,j)$, where $p=v_1\cdots v_m$ is a Schr\"oder path with $n-k$ steps $\longrightarrow$, one of which is $v_m$, and $j$ is an integer such that $v_j$ is a step ${\longrightarrow}$ or ${\searrow}$ which does not occur after the first step ${\longrightarrow}$ at height $0$.

By beginning a bead after each free vertex, improper tubings on $\Gamma_n$ form a Lyndon structure in bijection with festoons coloured by $(c_n)$, where $c_n$ counts tubings on $\mathcal{I}_n$ where only the final vertex is free. In particular, these tubings biject with pairs $(T,i)$, where $T$ is a tubing on $\mathcal{I}_n$ with the last vertex free and $i$ is an integer such that vertex $i$ of $\mathcal{I}_n$ does not occur after the first free vertex. By Lemma \ref{lem:schroder-bij}, these biject with the pairs $(p,j)$ defined above. 

Given $(p,j)$ we construct a Delannoy path: if $v_j$ is the final step, remove it to obtain a (strict) Schr\"oder path of length $2(n-1)$ with $n-k-1$ steps $\longrightarrow$. Otherwise, remove the final step and swap the subpaths occuring before and after $v_j$, so
\[v_1\cdots v_{j-1}v_jv_{j+1}\cdots v_m\mapsto v_{j+1}\cdots v_mv_jv_1\cdots v_{j-1}.\]
For the inverse map, if the Delannoy path is a strict Schr\"oder path, append a step $\longrightarrow$ to obtain a path $p$. Then $(p,m)$ satisfies the required conditions, where $v_m$ is final step of the path. Otherwise, choose $j$ maximal such that $v_{j+1}$ is a step whose height is minimal amongst all steps in the path. Swap the path before and after $v_j$ and append a step $\longrightarrow$ to obtain a path $p$, then take the pair $(p,j)$.
\end{proof}

We exemplify this three-way bijection: a tubing on $\Gamma_8$ with five tubes (with the top vertex distinguished), a Schr\"oder path with a marked step $j$, and a Delannoy path of length $14$ with two steps $\longrightarrow$.

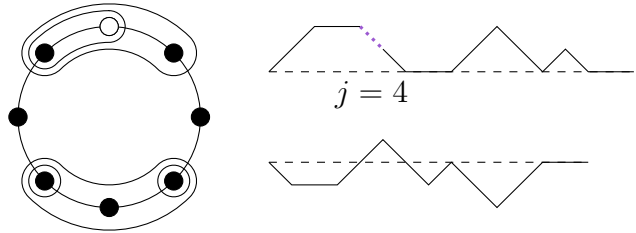
\[\begin{tikzpicture}[scale = 0.6]
\draw \tube{-2}{0}{2}{0.5}{1*45}{3*45};
\draw \tube{-2}{0}{2}{0.35}{2*45}{3*45};
\draw \tube{-2}{0}{2}{0.5}{5*45}{7*45};
\draw \tube{-2}{0}{2}{0.35}{5*45}{5*45};
\draw \tube{-2}{0}{2}{0.35}{7*45}{7*45};

\draw (-2,0) circle (2);
\foreach \ang in {0,1,2,3,4,5,6,7} {
	\draw [fill = black] ($(-2,0)+(45*\ang:2)$) circle (0.2);
}
\draw [fill = white] (-2,2) circle (0.2);

\draw (1.5,1) -- (2.5,2) -- (3.5,2);
\draw (4,1.5) -- (4.5,1) -- (5.5,1) -- (6.5,2) -- (7.5,1) -- (8,1.5) -- (8.5,1) -- (9.5,1);
\draw [thin, dashed] (1.5,1) -- (9.5,1);
\draw [very thick, dotted, color = ppl] (3.5,2) -- (4,1.5);
\node at (3.75,0.5) {$j=4$};

\draw (1.5,-1) -- (2,-1.5) -- (3,-1.5) -- (4,-0.5) -- (5,-1.5) -- (5.5,-1) -- (6,-1.5) -- (6.5,-2) -- (7.5,-1) -- (8.5,-1);
\draw [thin, dashed] (1.5,-1) -- (8.5,-1);
\end{tikzpicture}\]

\begin{proof}[Proof of Lemma \ref{lem:strict-schroder}]
By algebraic manipulation, we prove the equivalent statement
\[C(x)=x+\frac{C(x)^2}{1-C(x)}.\]
For $n=1$ we have only the empty path. For $n>1$, the first step $v_1={\nearrow}$ is eventually matched by the first $v_i={\searrow}$ at height $1$. We decompose the path into $k\geq 1$ parts separated by steps ${\longrightarrow}$ at height $1$ between $v_1$ and $v_i$, and a final part after $v_i$. For example:
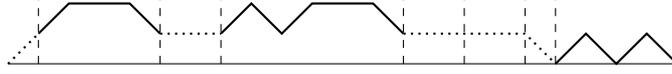
\[\begin{tikzpicture}[scale = 0.8]
\draw (0,0) -- (11,0);
\draw [thick, dotted] (0,0) -- (0.5,0.5);
\draw [thick] (0.5,0.5) -- (1,1) -- (2,1) -- (2.5,0.5);
\draw [thick, dotted] (2.5,0.5) -- (3.5,0.5);
\draw [thick] (3.5,0.5) -- (4,1) -- (4.5,0.5) -- (5,1) -- (6,1) -- (6.5,0.5);
\draw [thick, dotted] (6.5,0.5) -- (8.5,0.5) -- (9,0);
\draw [thick] (9,0) -- (9.5,0.5) -- (10,0) -- (10.5,0.5) -- (11,0);

\draw [thin, dashed] (0.5,0) -- (0.5,1.1);
\draw [thin, dashed] (2.5,0) -- (2.5,1.1);
\draw [thin, dashed] (3.5,0) -- (3.5,1.1);
\draw [thin, dashed] (6.5,0) -- (6.5,1.1);
\draw [thin, dashed] (7.5,0) -- (7.5,1.1);
\draw [thin, dashed] (8.5,0) -- (8.5,1.1);
\draw [thin, dashed] (9,0) -- (9,1.1);
\end{tikzpicture}\]
Given arbitrary paths for these parts, this construction introduces a step ${\nearrow}$, a step ${\searrow}$ and ${k-1}$ steps ${\longrightarrow}$. If the lengths of the subpaths are $n_1,\dots,n_{k+1}$, the length of the total path is now
\[2(n_1-1)+\cdots+2(n_{k+1}-1)+2k=2(n_1+\cdots+n_{k+1}-1).\]
So $c_1=1$ and for $n>1$ we have the recurrence
\[c_n=\sum_{k\geq 1}\sum_{n_1+\cdots+n_{k+1}=n}c_{n_1}\cdots c_{n_{k+1}}\]
which gives $C(x)=x+C(x)^2+C(x)^3+\cdots$ as required.
\end{proof}

\begin{proof}[Proof of Theorem \ref{thm:signed}]
Using the enumeration from the proof of Theorem \ref{thm:festoon-X} but taking into account the sign of each colour, we have
\[\#X_s^+-\#X_s^-=\sum_{s_1+s_2+\cdots}\rk(s_1)c_{s_1}c_{s_2}\cdots\]
which is the Gauss congruence associated with $(c_s)$. Fixed points of both signs are counted by
\[\#(X_s^+)^{\Cyc_d}=\sum_{t\in s/d}\begin{cases}\#X_t^++\#X_t^-&\text{if $d$ even},\\
\#X_t^+ &\text{if $d$ odd},\end{cases}\]
and
\[\#(X_s^-)^{\Cyc_d}=\sum_{t\in s/d}\begin{cases}0&\text{if $d$ even},\\
\#X_t^- & \text{if $d$ odd}.\end{cases}\]
Since $f_s(q)$ satisfies $q$-Gauss congruence,
\[f_s(\omega_d)=\sum_{t\in s/d}f_t(1)=\sum_{t\in s/d}\#X_{t}^+-\#X_{t}^-\]
for every $d\mid\rk(s)$, which equals $\#(X^+)^{\Cyc_d}-\#(X^-)^{\Cyc_d}$ when $\rk(s)$ is odd.
\end{proof}



\section*{Acknowledgements}
The author would like to thank Per Alexandersson, Ofir Gorodetsky, Eloise Little, Cory Peters and Ben Young for their thoughtful feedback and advice on this work.

The OEIS \cite{oeis} was an invaluable tool for finding and cross-checking information about integer sequences. Discovery and verification of examples was often performed with the computer-algebra program MAGMA \cite{magma}.


%
%

\bibliographystyle{alphaurl}
\bibliography{gauss-congruence}

\end{document}